\documentclass{amsart}
\usepackage{amscd, amssymb, amsmath, amsthm}
\usepackage{amsfonts,enumerate, latexsym}
\usepackage{graphics,graphicx,psfrag, mathrsfs}
\usepackage[colorlinks=true]{hyperref}

\newtheorem{theorem}{Theorem}[section]
\newtheorem{lemma}[theorem]{Lemma}
\newtheorem{proposition}[theorem]{Proposition}
\newtheorem{corollary}[theorem]{Corollary}
\newtheorem{conjecture}[theorem]{Conjecture}
\newtheorem{formula}[theorem]{Formula}

\theoremstyle{definition}
\newtheorem{definition}[theorem]{Definition}

\newtheorem{example}[theorem]{Example}
\newtheorem{question}[theorem]{Question}
\newtheorem{remark}[theorem]{Remark}

 \numberwithin{equation}{section}

\newcommand{\Rational}{{\mathbb Q}}

\newcommand{\Integral}{{\mathbb Z}}

\newcommand{\Hyp}{{\mathbf{H}^3}}

\newcommand{\id}{{\mathrm{id}}}

\newcommand{\pfib}{{\Phi}}
\newcommand{\periodic}{{\mathtt{per}}}
\newcommand{\pA}{{\mathtt{pA}}}

\newcommand{\gin}{{\mathbf{i}}} %geometric intersection number

\title{A characterization of virtually embedded
subsurfaces in $3$-manifolds}

\author[Yi Liu]{%
        Yi Liu} 
\address{%
    Mathematics 253-27\\
    California Institute of Technology\\
    Pasadena, CA 91125} 
\email{% 
    yliumath@caltech.edu}  

\thanks{Partially supported by NSF grant DMS-1308836}

\date{% 
 \today}

\begin{document}

\begin{abstract}
	The paper introduces the spirality character of the almost fiber part for
	a closed essentially immersed subsurface of a closed orientable aspherical
	3-manifold, which generalizes an invariant due to Rubinstein and Wang.
	The subsurface is virtually embedded 
	if and only if the almost fiber part is aspiral, and in this case,
	the subsurface is virtually a leaf of a taut foliation. 
	Besides other consequences,
	examples are exhibited that non-geometric $3$-manifolds with no Seifert fibered pieces
	may contain essentially immersed but not virtually embedded closed subsurfaces.
\end{abstract}

\maketitle

\section{Introduction}
	
	In this paper, we study virtually essentially embedded closed subsurfaces of 
	closed orientable $3$-manifolds. We always assume a closed
	surface or $3$-manifold to be connected, but a bounded surface or $3$-manifold
	may be disconnected.	
	A closed locally flatly immersed subsurface $S$ of a $3$-manifold $M$	
	is said to be \emph{essentially immersed},
	if $S$ is not a sphere and no homotopically nontrivial loop
	in $S$ is contractible in $M$. 
	A (homotopy) \emph{virtual property} of $S$ is a property which 
	holds for some lift of $S$ in some finite cover of $M$ up to homotopy.
	
	While all essentially immersed subsurfaces  of geometric $3$-manifolds
	are known to be virtually embedded, which is a consequence of Agol's proof 
	of the Virtual Haken Conjecture in the hyperbolic case \cite{Agol-VHC}
	and a simple fact in the non-hyperbolic case, cf.~\cite{NW}, 
	many essentially immersed subsurfaces  of 
	non-geometric $3$-manifolds do not to have this property.
	The first such example was discovered 
	by Rubinstein and Wang \cite{RW}. 
	In \cite{PW-embedded},  Przytycki and Wise have shown that
	a closed essentially immersed subsurface 
	of a $3$-manifold is virtually embedded if and only if it is $\pi_1$-separable.
	
	From a  topological point of view,
	we will characterize virtually essentially embedded 
	subsurfaces of $3$-manifolds in terms of the spirality 
	of the almost fiber part.
	Canonically induced from the JSJ decomposition,
	a closed essentially immersed subsurface $S$ 
	of a closed orientable aspherical $3$-manifold $M$
	can be decomposed into the \emph{JSJ subsurfaces} along \emph{JSJ curves}. 
	The \emph{almost fiber part} $\pfib(S)$ of the surface $S$
	is thus defined to be the union of all 
	the horizontal or geometrically infinite JSJ subsurfaces,
	carried by Seifert fibered or hyperbolic pieces of $M$ respectively,
	glued up along shared JSJ curves.
	Denote by $\Rational^\times$ the multiplicative abelian group of nonzero rational
	numbers. We will introduce 
	the \emph{spirality character} of the almost fiber part $\pfib(S)$,
		$$s:\,H_1(\pfib(S);\,\Integral)\to\Rational^\times,$$
	which is a natural invariant homomorphism.
	We say that $\pfib(S)$ is \emph{aspiral} if the spirality character only
	takes the values $\pm1$.
	When $M$ is a graph manifold and $S$ is orientable and horizontally
	immersed, the spirality character coincides with the invariant $s$ introduced 
	by Rubinstein--Wang \cite{RW}.
	
	In the case that $S$ is virtually essentially embedded, it is also interesting to 
	ask whether $S$ virtually fits into a taut foliation. 
	Recall that a codimension-$1$ foliation of a $3$-manifold is said to be \emph{taut}
	if for every leaf of the foliation there exists an immersed loop passing through the leaf
	and transverse to all leaves of the foliation. When $S$ is orientable,
	the question is equivalent to whether $S$ virtually minimizes the Thurston norm,
	cf.~\cite{Th-norm,Ga}.
		
	The main result of this paper is the following:
	
	\begin{theorem}\label{main-aspiralityCriterion}
		Let $M$ be an orientable closed aspherical $3$-manifold,
		and $S$ be a closed essentially immersed subsurface.
		Then following statements are equivalent:
		\begin{enumerate}
			\item $S$ is aspiral in the almost fiber part;
			\item $S$ is virtually essentially embedded;
			\item $S$ is virtually a leaf of a taut foliation.
		\end{enumerate}
	\end{theorem}
	
	Theorem \ref{main-aspiralityCriterion}
	implies, at least in principle,  a solution to the decision problem for (closed) surface subgroup separability
	in closed $3$-manifold groups. For example, 
	take the input a finite presentation of 
	a closed $3$-manifold group and a finite generating set of 
	a surface subgroup:
		
	\begin{corollary}\label{decideSeparability}
		There exists an algorithm to decide whether or not 
		a surface subgroup in a closed $3$-manifold group is separable.
	\end{corollary}
	
	Although the algorithm asserted by Corollary \ref{decideSeparability} is not
	practical for implementation, 
	the complexity of computation lies primarily in recognizing the almost fiber part 
	and understanding how the JSJ subsurfaces serve as virtual fibers.
	Once this piece of information is provided, the spirality character can be easily
	calculated, (Formula \ref{spiralityH}).
	Furthermore, the situation can be significantly improved 
	when the subsurface is immersed locally in a geometrically nice position.
	Especially, if the subsurface is an almost fiber
	transverse to a canonical suspension flow
	of a fibered cone,
	there is a formula for the spirality character 
	generalizing the work of Rubinstein--Wang
	most directly,
	(Formula \ref{formulaRW}).

	Let $\theta:F\to F$ be an orientation-preserving homeomorphism of
	an oriented closed surface $F$ of negative Euler characteristic.
	Denote by $M_\theta$ the mapping torus, which canonically fibers over 
	the circle with monodromy $\theta$.
	Inspired by the works \cite{RW,CLR-finiteFoliation,CLR-SIET},
	we investigate immersed subsurfaces $S$ transverse to 
	either the Seifert fibration or the pseudo-Anosov suspension flow
	in the JSJ pieces.
	
	%We say that a closed immersed subsurface $S$ of $M_\theta$
	%is \emph{JSJ parallel cutting}, if for each JSJ torus $T$, 
	%the JSJ curves carried by $T$ all cover one and the same essential
	%simple closed curve of $T$ up to homotopy.
	The following Corollary \ref{separableExamples} 
	provides a practical criterion
	for	aspirality without the
	necessity to understand the almost fiber part.
	See Subsection \ref{Subsec-surfaceAutomorphism} for further terminology
	about surface automorphisms.
		
	\begin{corollary}\label{separableExamples}
		Let $S$ be an oriented closed essentially immersed subsurface of $M_\theta$.
		Suppose all fractional Dehn twist coefficients of $\theta$  vanish.
		If $S$ can be homotoped to be	transverse to the JSJ tori and 
		to the Nielsen--Thurston suspension flow supported in the JSJ pieces,
		then $S$ is virtually embedded.
	\end{corollary}
		
	On the other hand, we are able to produce essentially immersed but not
	virtually embedded closed subsurfaces, under the circumstances complementary to
	Corollary \ref{separableExamples}:
	%The construction invokes techniques developed in \cite{PW-graph,PW-mixed,DLW}.
			
	\begin{corollary}\label{nonSeparableExamples}
		If $\theta$ has a nontrivial fractional Dehn twist coefficient,
		then $M_\theta$ contains a closed essentially immersed subsurface
		that is not virtually embedded.
	\end{corollary}
	
	In particular, $\pi_1(M_\theta)$ cannot be LERF or even surface subgroup separable 
	under the assumption of Corollary \ref{nonSeparableExamples}.
	Take the double of a pseudo-Anosov map of a bounded surface
	and compose with a nontrivial Dehn twist along 
	a doubling curve.
	Since the surface automorphism has a nontrivial fractional Dehn twist coefficient
	but does not have Seifert fibered pieces in the mapping torus,
	we find a counter-example to disprove \cite[Conjecture 9.1]{AFW-group}.
	Another consequence of Corollary \ref{nonSeparableExamples}, probably well known
	to experts, is that
	virtually fibered non-geometric graph manifolds 
	all fail to have surface subgroup separability of the fundamental group.
	In fact, one may also
	remove the virtually fibered condition by a slightly more careful construction
	--- compare \cite{Ne,NW,RW}. 	
	Suggested by these examples, we are inclined to propose the following:
	\begin{conjecture}\label{nonGeometricImpliesNonseparable}
		Every non-geometric closed aspherical $3$-manifold 
		contains a closed essentially immersed	subsurface which is not virtually embedded.
	\end{conjecture}
	
	The construction for Corollary \ref{nonSeparableExamples} relies on techniques
	developed by Przytycki--Wise \cite{PW-graph,PW-mixed,PW-embedded}, which
	can be used to produce
	closed or bounded virtually essentially embedded subsurfaces of $3$-manifolds 
	that intersect JSJ tori in certain controllable pattern.
	These so-called partial PW subsurfaces
	have been further investigated in \cite[Section 4]{DLW}.
	Understanding in details how such subsurfaces become virtually embedded
	leads to constructions crucial to the proof of
	Theorem \ref{main-aspiralityCriterion}.
	
	Theorem \ref{main-aspiralityCriterion} is inspired by previous works
	of various authors. In Przytycki--Wise \cite{PW-mixed}, we notice that an essentially immersed
	subsurface without geometrically infinite JSJ subsurfaces carried by hyperbolic pieces 
	is virtually embedded as long as the graph-manifold part 
	is virtually embedded, 
	and the obstruction to virtual embedding 
	lies	only in the horizontal part, known to be Rubinstein--Wang's invariant $s$.
	In Rubinstein--Wang \cite{RW} and Cooper--Long--Reid \cite{CLR-finiteFoliation,CLR-SIET}, 
	we notice that horizontally immersed subsurfaces of graph manifold and geometrically
	infinite subsurfaces in hyperbolic $3$-manifolds have certain similar behavior,
	especially if the $3$-manifold is fibered and the subsurface
	is transverse to a canonical suspension flow. 
	Hinted by the observations, we anticipate
	a uniform treatment for the almost fiber part, 
	and an extension of Przytycki--Wise's work 
	to handle the general situation.
	By intuition,
	it would be very surprising if a closed orientable essentially embedded subsurface failed
	to be virtually taut, or in other words, virtually Thurston norm minimizing.
	Having observed that fibers or retracts of JSJ pieces
	are taut subsurfaces, we are able to strengthen
	the argument for virtual embedding
	to show virtual tautness.	
	
	In Section \ref{Sec-prelim}, we review preliminary materials in $3$-manifold
	topology and introduce some terminology. 
	In Sections \ref{Sec-theFirstRecurrenceMap}, \ref{Sec-theSpiralityCharacter}, we define and study
	the spirality character of the almost fiber part; in Sections \ref{Sec-separability}, \ref{Sec-virtualTautness},
	we prove the main result Theorem \ref{main-aspiralityCriterion}.
	In Sections \ref{Sec-flowTransverse}, \ref{Sec-application}, we introduce pseudo graph manifolds,
	a generalization of graph manifolds containing atoroidal pieces equipped with
	pseudo-Anosov suspension flow structures.
	We derive a practical formula to calculate the spirality character for subsurfaces 
	transverse to flow in pseudo graph manifolds, and prove
	Corollaries \ref{separableExamples}, \ref{nonSeparableExamples}.
	In Section \ref{Sec-decisionProblems}, we address the decision problem for surface subgroup separability
	using the spirality character, and  prove Corollary \ref{decideSeparability}.
	In Section \ref{Sec-conclusions}, we propose a few related questions.
	
	\bigskip\noindent\textbf{Acknowledgement}. The author thanks Yi Ni and Alan Reid
	for valuable communications. The author would also like to thank the anonymous referee for
	pointing out some incorrect formulations in Proposition \ref{psiAlpha} and Formula \ref{formulaRW}
	in an earlier manuscript.

\section{Preliminaries}\label{Sec-prelim}
	In this section, we review JSJ decomposition of $3$-manifolds \cite{JS,Joh,Th-book} 
	and the Nielsen--Thurston classification of 
	surface automorphisms \cite{Ni,Th-NT}.
	We refer to the survey \cite{AFW-group} for topics related to virtual properties of $3$-manifolds.
	
	\subsection{JSJ decomposition}\label{Subsec-JSJDecomposition}
		
		Let $M$ be a compact orientable irreducible $3$-manifold.
		The Jaco--Shalen--Johanson (JSJ) decomposition
		asserts that $M$ can be cut
		along a minimal collection
		of essential tori, canonical up to isotopy,
		into Seifert fibered pieces and atoroidal pieces.
		We call the minimal collection of tori the \emph{JSJ tori}, and
		the pieces the \emph{JSJ pieces}. The JSJ decomposition
		is a graph-of-spaces decomposition of $M$ whose vertices and edges
		naturally correspond to the JSJ pieces and the JSJ tori, respectively.
		It induces a graph-of-group decomposition of $\pi_1(M)$ up to natural
		isomorphism with respect to choices of base points.
		
		\subsubsection{Structure of JSJ pieces}
		We say that a compact orientable irreducible $3$-manifold is \emph{nonelementary} if its fundamental group
		is not virtually solvable. By Hyperbolization,
		a nonelementary atoroidal JSJ piece is $\Hyp$-geometric, and the
		hyperbolic structure is unique up to isometry. 
		A nonelementary Seifert fibered JSJ piece is a circle bundle over a hyperbolic $2$-orbifold,
		and the Seifert fibration structure is canonical up to isotopy.
		When passing to a finite cover of $M$, every nonelementary JSJ piece elevates 
		to be JSJ pieces of the cover.
		On the other hand,
		if $M$ is not itself elementary, any elementary JSJ piece
		of $M$ is homeomorphic to the \emph{orientable thickened Klein bottle}, namely,
		the (characteristic) twisted compact interval bundle over a Klein bottle whose total space is orientable.
		When passing to a finite cover of $M$,
		an elementary JSJ piece may elevate to be a regular neighborhood
		of a JSJ torus. In fact, if $M$ is not elementary, there is always a double cover of
		$M$ in which there are no elementary JSJ pieces.
		
		\subsubsection{Induced decomposition on subsurfaces}
		Let $M$ be an orientable closed aspherical $3$-manifold,
		and $S$	be a closed essentially immersed subsurface.
		Up to homotopy, $S$ intersects the JSJ tori of $M$
		in a minimal finite (possibly empty) collection of disjoint essential 
		curves of $S$, called the \emph{JSJ curves},
		and the complementary components of the union of the JSJ curves
		are essential subsurface of $S$, called the \emph{JSJ subsurfaces}.
		The decomposition of $S$ above induced from the JSJ decomposition of 
		$M$ is unque up to isotopy,
		which we will refer to as the \emph{induced JSJ decomposition}.
		If $S$ is	nonelementary, every JSJ subsurface
		in an elementary JSJ piece is an essential sub-band of $S$.
		A JSJ subsurface in a nonelementary atoroidal JSJ piece
		is either geometrically finite or geometrically infinite
		with respect to the canonical hyperbolic
		geometry.	A JSJ subsurface in a nonelementary Seifert fibered piece
		is either vertical or horizontal with respect to 
		the canonical Seifert fibration. It is now a known fact that
		any JSJ subsurface can be lifted to be embedded in some finite cover
		of the carrying JSJ piece.
		
		A properly embedded subsurface of a compact orientable $3$-manifold is 
		said to be an (ordinary or semi) \emph{fiber}, if the $3$-manifold
		is a surface bundle over a $1$-orbifold and the subsurface
		is isotopic to a fiber. The base $1$-orbifold can be a circle or an interval
		if the subsurface is orientable, or a semi-circle or a semi-interval
		if the subsurface is nonorientable.
		A \emph{virtual fiber} is a properly immersed subsurface which can be lifted
		into some finite cover the $3$-manifold to become a fiber up to homotopy
		relative to boundary.
		In particular, we have three types of JSJ subsurfaces which
		are virtual fibers, namely, JSJ subsurfaces in an elementary piece, or 
		geometrically infinite in a nonelementary atoroidal piece, or
		horizontal in a nonelementary Seifert fibered piece.
		 
		\begin{definition}
		In the dual graph $\Lambda$ of the induced JSJ decomposition of $S$,
		the vertices that correspond to the virtual-fiber JSJ subsurfaces of $S$
		span a complete subgraph $\pfib(\Lambda)$ of $\Lambda$.
		The \emph{almost fiber part} of $S$ is defined to be the essential subsurface
			$$\pfib(S)\,\subset\,S$$
		dual to the complete subgraph $\pfib(\Lambda)$.
		\end{definition}
		
	\subsection{Classification of surface automorphisms}\label{Subsec-surfaceAutomorphism}
		Let $F$ be an oriented closed surface of negative Euler characteristic, and
		$\theta:F\to F$ be an \emph{automorphism}, namely, an orientation-preserving
		homeomorphism.
		By the Nielsen--Thurston classification of surface automorphisms,
		there exists a unique minimal collection of essential curves up to isotopy,
		called the \emph{canonical reduction system}, such that
		$\theta$ can be isotoped to preserve a regular neighborhood of the union of the reduction
		curves, and the first return map to each complementary components
		is either periodic or pseudo-Anosov up to isotopy. 
		Hence the compact closures	of the complement components 
		form two canonical subunions invariant under $\theta$, namely, 
		the \emph{periodic part} and the \emph{pseudo-Anosov part}.
		We denote by $\theta_\periodic$ and $\theta_\pA$ the periodic or pseudo-Anosov
		automorphisms on these parts respectively, which are unique up to topological
		conjugacy and freely isotopic to the restrictions of $\theta$.
				
		\subsubsection{Suspension flows}
		Denote by $M_\theta$ the mapping torus of $\theta$, namely,
				$$M_\theta\,=\,\frac{F\times[0,1]}{(x,0)\sim (\theta(x),1)}.$$
		The JSJ pieces of $M$ are equipped with a canonical flow suspending $\theta_\periodic$
		or $\theta_\pA$, which we will call the \emph{Nielsen--Thurston suspension flow}.
		On each boundary component of a JSJ piece, there exists a closed leaf and different closed leaves
		are parallel to each other. We will call any such leaf as a \emph{degeneracy slope} of the JSJ
		piece, borrowing a term from \cite{GaO}.		
		On a Seifert fibered piece, this flow is exactly the Seifert fiberation, and the degeneracy slope
		is an ordinary fiber.
		On an atoroidal piece, the flow is a pseudo-Anosov flow, which preserves two singular
		codimension-$1$ foliations transverse along flow lines coming from the suspension
		of the stable and unstable foliations of $\theta_\pA$.
		A JSJ torus $T$ of $M$ receives two degeneracy slopes from the JSJ pieces on both sides,
		and they do not match with each other in general.
		
		\subsubsection{Fractional Dehn twist coefficients}
		We introduce a notion of fractional Dehn twist coefficients which generalizes a definition
		of Honda--Kazez--Mati\'{c} \cite{HKM}.
		Given a surface automorphism $\theta:F\to F$ as above, a sufficiently high power $\theta^m$
		preserves each reduction curve $e_1,\cdots,e_r$, and all the periodic points
		of the restrictions of $\theta_\periodic^m$ and $\theta_\pA^m$ on $e_1,\cdots,e_r$ are fixed points.
		Since the restriction of $\theta_\pA^m$ on the boundary of the pseudo-Anosov part
		is isotopic to the identity relative to the fixed points,
		we may extend by the identity and
		regard $\theta_\pA^m$ as an automorphism of $F$ up to isotopy.
		Denote by $D(e_i)$ the (right-hand) Dehn twist of $F$ along $e_i$.
		There is a factorization of $\theta$ into a commutative product
		\begin{equation}
			\theta^m\,=\,D(e_1)^{k_1}\cdots D(e_r)^{k_r}\theta_\pA^m,
		\end{equation}
		up to isotopy, where $k_i$ are  integers uniquely determined by $\theta$ and $m$.
		We define the \emph{fractional Dehn twist coefficient} of $\theta$ along $c$ to be
		\begin{equation}
			\mathbf{c}(\theta,e_i)\,=\,k_i\,/\,m.
		\end{equation}
		The definition does not depend on the choice of $m$,
		and indeed, $\mathbf{c}(\theta^n,e_i)$ equals $n\cdot\mathbf{c}(\theta,e_i)$.
		For a reduction curve $e_i$ carried by a JSJ torus $T$ of $M_\theta$,
		the fractional Dehn twist coefficient of $\theta$ vanishes
		along $e_i$ exactly when
		the degeneracy slopes induced from both sides of $T$ match up with each other.
		Of course,
		this vanishing occurs only if the reduction curve is adjacent to
		the pseudo-Anosov part on at least one side.

\section{Spirality of loops in the almost fiber part}\label{Sec-theFirstRecurrenceMap}
	We start our discussion with a descriptive definition of spirality,
	from a dynamical perspective.
	Let $S$ be a closed essentially immersed subsurface  
	of an orientable closed aspherical $3$-manifold $M$.
	
	For a JSJ curve $c$ contained in the almost fiber part $\pfib(S)$, 
	fix a	basepoint $*$ in $c$. Then $M$ has
	an induced basepoint $x_0$, and 
	the subgroup $\pi_1(S,*)$ of $\pi_1(M,x_0)$ 
	gives rise to a covering space $\tilde{M}_S$
	of $M$ with a lifted basepoint $\tilde{x}_0$.
	%and a based embedded lift of $S$.
	Denote by $A_c$ the discrete subset of the JSJ
	cylinder $\tilde{T}$ that carries the based lift of $c$,
	such that $A_c$ consists of the lifts of $x_0$ 
	in $\tilde{T}$.
	Since $\tilde{T}$ is a finitely generated rank-$1$ abelian cover
	of the JSJ torus $T$ of $M$ carrying $c$,
	$A_c$ can be identified with the deck transformation group,
	with $\tilde{x}_0$ the trivial element.
	For any closed path $\alpha:[0,1]\to \pfib(S)$ based at $*$, 
	the lifts of $\alpha$ based in $A_c$ naturally induce a map 
		$$\psi_\alpha:\,A_c\to\tilde{M}_S,$$
	taking any point to the terminal endpoint of the lift of $\alpha$
	based at that point.
	It turns out that the subset of $A_c$:
		$$D_\alpha(A_c)\,=\,\psi^{-1}_\alpha(A_c)\cap A_c$$
	contains a finite index subgroup of $A_c$.
	We will see that the system of restrictions of $\psi_\alpha$ to $D_\alpha(A_c)$:
		$$(\psi_\alpha,\,D_\alpha(A_c))$$
	gives rise to a partial action of $\pi_1(\pfib(S),*)$ on $A_c$.
	Moreover, we have the following description of 
	the partial transformations $\psi_\alpha$ on $A_c$ as partial dilatations.
	See Subsection \ref{Subsec-partialDilatations}
	for definitions and elementary properties.
		
	\begin{proposition}\label{psiAlpha}
		Based lifting	induces a canonical partial action of $\pi_1(\pfib(S),*)$ 
		on $A_c$ by partial dilatations.
	\end{proposition}
	
	\begin{definition}\label{spiralityOfLoops}
		The \emph{spirality} of a closed path $\alpha:[0,1]\to\pfib(S)$ based at $*$
		is defined to be the dilatation rate of the partial transformation $\psi_\alpha$ on $A_c$, denoted as
			$$s(\alpha)\,=\,\lambda(\psi_\alpha).$$
	\end{definition}
	
	\begin{remark}
		We will extend the definition to $1$-cycles of $\pfib(S)$ (Definition \ref{spiralityCharacter}).
	\end{remark}
	
	The partial dilatation description of the partial transformations 
	reveals certain asymptotic nature of 
	the spirality. To visualize, let us say that a subsegment of an elevation of $\alpha$
	is spiral if it intersects $A_c$ in a geometric sequence of points
	$a,ar,\cdots,ar^m$, where $r$ is the spirality of $\alpha$.
	Then it can be implied from the definition of partial dilatations
	(Definition \ref{pseudoDilatation}) that 
	there are no bi-infinitely spiral elevations of $\alpha$
	in $\tilde{M}_S$, but there are
	always elevations with arbitrarily long spiral segments. When the spirality
	$s(\alpha)$ is a positive integer, many elevated lines of $\alpha$ in $\tilde{M}_S$ 
	are infinitely spiral in one direction toward an end of $\tilde{T}$. 

	We prove Proposition \ref{psiAlpha} in the rest of this section.
	
	\subsection{Partial dilatations}\label{Subsec-partialDilatations}
		Let $A$ be a finitely generated abelian group of rank at least $1$.
		
		\begin{definition}\label{pseudoDilatation}
			A setwise map $\psi:D(A)\to A$ is called
			a \emph{partial dilatation} of $A$ if 
			there exist a pair of nonzero integers
			$p$, $q$ such that $D(A)$ contains $qA$, and
				$$\psi(qv)\,=\,pv.$$
			holds for all $v\in A$.
		\end{definition}
				
		\begin{lemma}
			The ratio $p/q$ of a partial dilatation $\psi$ does not depend on the choice of defining integers $p$, $q$.
			%Moreover, any injective extension of a partial dilatation is also a partial dilatation with the same
			%ratio.
		\end{lemma}
		
		\begin{proof}
			Since $D(A)$ contains a finite index subgroup of $A$, we may choose an element $v\in D(A)$
			of infinite order. For any defining pairs of integers $(p,q)$ and $(p',q')$,
			we have $\psi(qq'v)$ equals $pq'v$ and $p'qv$, so $p/q$ equals $p'/q'$.
			%Suppose that $\psi'$ is an extension of a partial dilatation $\psi:D(A)\to A$ with defining integers $p$, $q$.
			%As $D(A)$ contains a finite index subgroup $B$ of $A$, 
			%we may take the defining integers $[A:B]p$, $[A:B]q$ to see
			%that $\psi'$ is a partial dilatation with the same ratio.
		\end{proof}
				
		\begin{definition}
			We define the \emph{dilatation rate} of a partial dilatation $\psi$ of $A$ to be the nonzero rational
			value 
				$$\lambda(\psi)=p/q.$$			
		\end{definition}
		
		Recall that a right \emph{partial action} of group $G$ on set $X$ 
		is a collection of pairs $(\psi_g,D_g)$ for all $g\in G$, 
		where $D_g$ is a subset of $X$ and $\psi_g:D_g\to D_{g^{-1}}$ is a bijection,
		satisfying that $(\psi_{\id},\,D_{\id})$ is $(\id_X,\,X)$ and that
		$D_{gh}$ contains $D_g\cap\psi_{g^{-1}}(D_h)$ and
		$\psi_{gh}$ is an extension of the composed action $\psi_g$ followed by $\psi_h$,
		for all $g,h\in G$.
		
		Suppose that $G$ is a group that partially acts on $A$ by partial dilatations $(\psi_g,D_g)$.
		Then it is obvious that the dilatation rate function
			$$\lambda:\,G\to\,\Rational^\times$$
		which assigns $\lambda(\psi_g)$ to any element $g\in G$ is a group homomorphism.		
		
	\subsection{The partial action by partial dilatations}
		Continue to adopt the notations introduced at the beginning of this section.
				
		\begin{lemma}\label{recurrence}
			The system $(\psi_\alpha,D_{\alpha}(A_c))$ 
			for all $\alpha\in\pi_1(\pfib(S),*)$ defines a right partial action
			of $\pi_1(\pfib(S),*)$ on $A_c$.			
		\end{lemma}
		
		\begin{proof}
			For any closed path $\alpha:[0,1]\to \pfib(S)$ based at $*$,
			recall that the domain $D_{\alpha}(A_c)$ consists of all the lifts of $x_0$ in $A_c$
			based at which the lift of $\alpha$ arrives into $A_c$. By uniqueness of lifting,
			we have that $\psi_\alpha:D_{\alpha}(A_c)\to D_{\bar{\alpha}}(A_c)$ is bijective,
			where $\bar{\alpha}$ is the reversal of $\alpha$. 
			Clearly the trivial loop induces the identical transformation on $A_c$.
			Composition of partial transformations
			$\psi_\alpha$, whenever possible, is induced by lifting of concatenations of loops.
			This shows that $\pi_1(\pfib(S),*)$ partially acts on $A_c$ from the right.
		\end{proof}
		
		In order to see that $\pi_1(\pfib(S),*)$ partially acts on $A_c$ by partial dilatations,
		we factorize any partial transformation
		$\psi_\alpha$ into the composition of a sequence of partial maps
		as follows.
		%
				%
				%
		%\begin{lemma}\label{recurrence}
			%For any closed path $\alpha:[0,1]\to \pfib(S)$ based at $*$, and for any point
			%$\tilde{x}\in A_c$, denote by
			%$\tilde{\alpha}_{\tilde{x}}:[0,1]\to \tilde{M}_S$ the lift of $\alpha$ based
			%at $\tilde{x}$. Then $\tilde{\alpha}_{\tilde{x}}(1)$ lies in $A_c$.			
		%\end{lemma}
		%
		%\begin{proof}
			%Let $\Phi_0$ be the component of $\pfib(S)$ containing $*$.
			%Consider the minimal union $W$ of the elevations of JSJ pieces in $\tilde{M}_S$,
			%so that $W$ contains the lift of $\Phi_0$ based at $\tilde{x}_0$. 
			%As the JSJ subsurfaces of $\Phi_0$ are all virtual fibers of the carrier JSJ pieces
			%of $M$, it is easy to see
			%that the elevated JSJ pieces $W$ are homeomorphic a line bundle over $S_v$
			%corresponding to the JSJ subsurfaces $S_v$ of $\Phi_0$,
			%and the elevated JSJ tori of $W$ are infinite cylinders corresponding to the JSJ curves
			%in $\Phi_0$. This implies that 
			%$\tilde{\alpha}_{\tilde{x}}(1)\in A_c$ for any $\tilde{x}\in A_c$.
		%\end{proof}
		%
		%It follows from Lemma \ref{recurrence} that
		%the \emph{first recurrence map} induced	by $\alpha$
			%$$\psi_\alpha:A_c\to A_c,$$
		%which takes any $\tilde{x}$ to $\tilde{\alpha}_{\tilde{x}}(1)$,
		%is well defined. Furthermore,
		%the map depends only on the homotopy class of $\alpha$ in $\pi_1(\pfib(S),*)$,
		%and composition of two first recurrence maps $\psi_\beta$, $\psi_\alpha$ is 
		%the first recurrence map $\psi_{\alpha\beta}$. The first recurrence
		%is indeed a right action of $\pi_1(\pfib(S),*)$ on the set $A_c$.  
				%

		Possibly after a homotopy relative to the base point,
		we may assume that the lifted closed path $\tilde\alpha:[0,1]\to \tilde{M}_S$ 
		of $\alpha$ based at $\tilde{x}_0$ intersects the JSJ curves 
		of $\pfib(S)$ consecutively in 
			$$y_0,\cdots,y_n,$$ 
		with $y_0$ and $y_n$ equal to $\tilde{x}_0$.
		Let 
			$$\alpha_1,\cdots,\alpha_n$$
		be the subpaths $\alpha$ so that
		the corresponding subpaths $\tilde{\alpha}_i$ of $\tilde{\alpha}$
		are from $y_{i-1}$ to $y_i$, respectively.
		Note that all the $y_i$ lie on the lift of $S$ at $\tilde{x}_0$. 
		Since each $y_i$ lies on a unique lifted JSJ curve $c_i$, 
		we denote by $\tilde{T}_i$ the JSJ cylinder
		of $\tilde{M}_S$ corresponding to $c_i$, not necessarily mutually distinct.
		Let $A_i$ denote the deck transformation group of $\tilde{T}_i$ which covers
		the underlying JSJ torus of $M$. We may identify each $A_i$ as the orbit
		of $y_{i-1}$, so that $a\in A_i$ corresponds to $a.y_i$.
		For each $\alpha_i$, we obtain an induced map 
			$$\psi_i:D_i(A_{i-1})\to A_i.$$ 
		The domain $D_i(A_{i-1})$ is the largest subset of $A_{i-1}$
		such that the lift of $\alpha_i$ based at any point $y\in D_i(A_{i-1})$ 
		has its terminal endpoint contained in $A_{i-1}$,
		and terminal endpoint is defined to be $\psi_i(y)$. 
		To facilitate the notation, we regard $(\psi_i,D_i(A_{i-1}))$ as a partial map of $A_{i-1}$ to $A_i$,
		denoted as
			$$\psi_i:A_{i-1}\dashrightarrow A_i,$$
		with the domain implicitly assumed. The composition of two partial maps
		is the partial map defined
		on the largest domain so that the composition is admissible in the obvious sense.
		With these notations, the maps $\psi_i$ give rise to a restricted factorization of $\psi_\alpha$:
		$$A_0\stackrel{\psi_1}{\dashrightarrow}A_1\stackrel{\psi_2}{\dashrightarrow}\cdots
		\stackrel{\psi_n}{\dashrightarrow} A_n,$$
		where $A_0$ and $A_n$ are both $A_c$, in other words,
		$\psi_\alpha$	is an extension of 
		the composition of $\psi_i$ with possibly larger domain.

		It would be convenient to extend the definition of partial dilatations to partial maps
		between different groups. We say that a partial map
			$$\psi:A\dashrightarrow A'$$
		between two finitely generated rank-$1$ abelian groups is a \emph{partial dilatation}
		if there exist infinite-order elements $v\in A$ and $v'\in A'$, so that $D(A)$ contains $\Integral v$, and
			$$\psi(mv)=mv'.$$
		holds for all $m\in\Integral$.
		Note that the image of $\psi$ contains $\Integral v'$,
		so the composition of two partial dilatations is still a partial dilatation.
		When $A$ equals $A'$, it is easy to check that the new definition 
		agrees with the definition before.
		
%
		%Factorize any first recurrence map $\psi_\alpha:A_c\to A_c$ as a sequence of maps
		%as follows. Possibly after a homotopy relative to the base point,
		%we may assume that the lifted closed path $\tilde\alpha:[0,1]\to \tilde{M}_S$ 
		%of $\alpha$ based at $\tilde{x}_0$ intersects the JSJ curves 
		%of $\pfib(S)$ consecutively in 
			%$$y_0,\cdots,y_n,$$ 
		%with $y_0$ and $y_n$ equal to $\tilde{x}_0$.
		%Let 
			%$$\alpha_1,\cdots,\alpha_n$$
		%be the subpaths $\alpha$ so that
		%the corresponding subpaths $\tilde{\alpha}_i$ of $\tilde{\alpha}$
		%are from $y_{i-1}$ to $y_i$, respectively.
		%Note that all the $y_i$ lie on the lift of $S$ at $\tilde{x}_0$. 
		%Since each $y_i$ lies on a unique lifted JSJ curve $c_i$, 
		%we denote by $\tilde{T}_i$ the JSJ cylinder
		%of $\tilde{M}_S$ corresponding to $c_i$, not necessarily mutually distinct.
		%Let $A_i$ denotes the deck transformation group of $\tilde{T}_i$ which covers
		%the underlying JSJ torus of $M$. We may identify each $A_i$ as the orbit
		%of $y_{i-1}$, so that $a\in A_i$ corresponds to $a.y_i$.
		%For each $\alpha_i$, we obtain an induced map $\psi_i:A_{i-1}\to A_i$ 
		%that takes any point $y\in A_{i-1}$ to the terminal endpoint of the lift of $\alpha_i$
		%based at $y$. These maps $\psi_i$ give rise to a factorization of $\psi_\alpha$:
		%$$A_0\stackrel{\psi_1}{\longrightarrow}A_1\stackrel{\psi_2}{\longrightarrow}\cdots
		%\stackrel{\psi_n}{\longrightarrow} A_n,$$
		%where $A_0$ and $A_n$ are both $A_c$.
		%
		\begin{lemma}\label{eachPseudoDilatation}
			Each partial map $\psi_i:A_{i-1}\dashrightarrow A_i$ is a partial dilatation.
		\end{lemma}
		
		\begin{proof}
			Let $S_v$ be the JSJ subsurface of $\pfib(S)$ containing the subpath $\alpha_i$.
			
			First assume that $S_v$ is orientable. 
			Denote by $\tilde{J}_v$ the induced JSJ piece of $\tilde{M}_S$
			which intersects the based lift of $S$ in a lift of $S_v$.
			As $S_v$ is a virtual fiber of the carrier JSJ piece $J_v$ of $M$,
			$\tilde{J}_v$ is a virtually infinite cyclic cover of $J_v$.
			Thus there exists a deck transformation
			$\tau:\tilde{J}_v\to \tilde{J}_v$ of infinite order. 
			Moreover, possibly after replacing $\tau$
			with a nontrivial power, we may assume $\tau$ to preserve each boundary component
			of $\tilde{J}_v$.
			In particular,
			$\tau$ acts on the JSJ cylinders $\tilde{T}_{i-1}$ and $\tilde{T}_i$
			by infinite-order deck transformations, which are affine translations 
			by adding $a_{i-1}\in A_{i-1}$ and $a_i\in A_i$ respectively.
			In other words, the domain $D_i(A_{i-1})$ of $\psi_i$ contains $\Integral a_{i-1}$, and
				$$\psi_i(ma_{i-1})=ma_i$$
			for all $m\in\Integral$.
			%
			%For any $u\in D_i(A_{i-1})$, let $\tilde{\alpha}_u$ be the lift of $\alpha$ in $J_v$
			%based at $u.y_{i-1}$. The terminal endpoint of $\tilde{\alpha}_u$ is the point $\psi_i(u).y_i$,
			%which lies in $A_i$ by the definition of $D_i(A_{i-1})$. 
			%Then $\tau.\tilde{\alpha}_u$ is the elevation of $\alpha$
			%based at $(u+a_{i-1}).y_{i-1}$ with the terminal endpoint $(\psi(u)+a_i).y_i$. 
			%This means $\psi_i(u+a_{i-1})=\psi_i(u)+a_i$ for all $u\in A_{i-1}$.
			%Note that $\psi_i:D_i(A_{i-1})\to A_i$ is certainly injective.
			Therefore, $\psi_i$ is a partial dilatation.	
			
			If $S_v$ is not orientable, we may take the characteristic orientable double cover $S^*_v$,
			namely, the double cover corresponding to
			the kernel of the orientation character $\omega:\pi_1(S_v)\to\{\pm1\}$.
			The pullback cover $\tilde{J}^*_v$
			of $\tilde{J}_v$ is equipped with a free involution $\nu$.
			As before, we may find an infinite-order deck transformation 
			$\tau^*$ for $\tilde{J}^*_v$ componentwise preserving the boundary. 
			Observe that each component of $\partial \tilde{J}_v$ has two disjoint
			lifts in $\partial \tilde{J}^*_v$ 
			that are exchanged under $\nu$, so the action of $\tau^*$ commutes with $\nu$ 
			on $\partial\tilde{J}^*_v$. Thus $\tau^*$ descends to become an affine translation
			$\tau$ acting on $\partial\tilde{J}_v$, and in particular, on $\tilde{T}_{i-1}$ and
			$\tilde{T}_i$. We may proceed to complete the proof by the same argument as
			of the orientable case.
		\end{proof}
		
		By Lemma \ref{eachPseudoDilatation}, the composition of the partial dilatations $\psi_1,\cdots,\psi_n$
		is also a partial dilatation. Since $\psi_\alpha$ is an extension of the composition,
		$\psi_\alpha$ is a partial dilatation as well.		
		This completes the proof of Proposition \ref{psiAlpha}.

\section{The spirality character of the almost fiber part}\label{Sec-theSpiralityCharacter}
	
	In this section, we take a constructive approach and
	recapture the spriality of the almost fiber part 
	via a canonically associated $\Rational^\times$-principal bundle. 
	Here $\Rational^\times$ denotes the multiplicative group
	of nonzero rational numbers endowed with the discrete topology.
	The holonomy of a $\Rational^\times$-principal bundle 
	induces a so-called spirality character in the first cohomology
	of the base space in $\Rational^\times$ coefficient.
	See Subsection \ref{Subsec-theSpiralityCharacter} for details.
	The construction is based on a flexible auxiliary choice of fibered covers
	of the JSJ pieces which carry JSJ subsurfaces of the almost fiber part,
	and the setting will be convenient for calculation.
		
	Let $M$ be an orientable closed $3$-manifold,
	and $S$ be a closed essentially immersed subsurface of $M$.

	\begin{proposition}\label{bundleH}
		There exists a principal $\Rational^\times$-bundle over the almost fiber part:
				$$\mathscr{H}\to\pfib(S)$$
		naturally induced by the immersion up to isomorphism 
		of principal $\Rational^\times$-bundles, whose spirality character (Definition \ref{spiralityCharacterDefinition})
		satisfies
				$$\langle s(\mathscr{H}),\,[\alpha]\rangle\,=\,s(\alpha),$$
		for any closed path $\alpha$ in $\pfib(S)$ with a basepoint on a JSJ curve.
		Moreover, the sign reduction of the spirality character of $\mathscr{H}$ 
		is the orientation character of $\pfib(S)$.
	\end{proposition}
	
	Here the naturality simply means that for any finite covering map $S'\to S$,
	the pullback of $\mathscr{H}\to\pfib(S)$ 
	is isomorphic to $\mathscr{H}'\to\pfib(S')$.

	\begin{definition}\label{spiralityCharacter}
		The \emph{spirality character of the almost fiber part} of $S$ is defined to be 
			$$s(\mathscr{H})\,\in\,H^1(\pfib(S);\,\Rational^\times),$$
		namely, the spirality character
		of the associated principal $\Rational^\times$-bundle $\mathscr{H}$.
		For any integral $1$-cycle $\alpha$ of $\pfib(S)$,
		the \emph{spirality} of $\alpha$ is defined to be 
			$$s(\alpha)= \langle s(\mathscr{H}),\,[\alpha]\rangle,$$
		where $[\alpha]\in H_1(\pfib(S);\Integral)$ is the represented homology class.
	\end{definition}
	
	We prove Proposition \ref{bundleH} in the rest of this section.
	
	\subsection{The spirality character}\label{Subsec-theSpiralityCharacter}
	
	%\subsubsection{The spirality class of a principal $\Rational^\times$-bundle}
	\begin{definition}\label{spiralityCharacterDefinition}
	For a principal $\Rational^\times$-bundle $P$ over a CW	space $X$, 
	we define the \emph{spirality character} of $P$, denoted as
		$$s(P)\in H^1(X;\Rational^\times),$$
	by the holonomy of $P$.
	\end{definition}
	
	In other words, for any closed path $\alpha:[0,1]\to X$,
	after specifying a lift of the initial point, 
	there is a unique lift $\tilde{\alpha}:[0,1]\to P$ of $\alpha$.
	The ratio $\tilde{\alpha}(1)\,/\,\tilde{\alpha}(0)$ valued in $\Rational^\times$
	depends only on $[\alpha]\in H_1(X;\Integral)$ and induces 
	a homomorphism of abelian groups
		$$s(P):\,H_1(X;\Integral)\to\Rational^{\times}.$$
	We may alternatively regard $s(P)$ as an element of $H^1(X;\Rational^\times)$.
	
	The spirality character is a (multiplicatively valued) characteristic class of principal $\Rational^\times$-bundles in the
	sense that it is natural with respect to pullbacks. Under the sign homomorphism 
	$\mathrm{sgn}:\Rational^\times\to \{\pm1\}$,
	the induced homomorphism 
		$$\mathrm{sgn}:\,H^1(X;\,\Rational^\times)\to H^1(X;\,\{\pm1\})$$
	reduces the spirality character of $P$ to the orientation character of 
	the associated principal $\{\pm1\}$-bundle 
	$\mathrm{sgn}(P)=P\times_{\mathrm{sgn}}\{\pm1\}$:
	\begin{equation}
		\mathrm{sgn}(s(P))=\omega(\mathrm{sgn}(P)).
	\end{equation}
	Here $\omega$ denotes the orientation character, which is just
	the first Stiefel--Whitney class (of an associated real vector bundle),
	except with the additive coefficient group $\Integral/2$
	treated as the multiplicative group $\{\pm1\}$ instead.
	
	\begin{definition}
	We say that $P$ is \emph{aspiral} if $s(P)$ is torsion,
	or equivalently, only takes values $\pm1$ for all $[\alpha]\in H_1(X;\Integral)$.
	\end{definition}
	
	A general formula to calculate the spirality character will be given based on some 
	topological setup data 
	that we choose to construct of the associated principal bundle, (Formula \ref{spiralityH}).
	
	\subsection{The associated principal bundle $\mathscr{H}$}\label{Subsec-bundleH}
		
		\subsubsection{Setup}
		As $S$ has induced decomposition by the JSJ curves,
		denote by $\Lambda$ and $\pfib(\Lambda)$ the dual graphs of $S$ and $\pfib(S)$, respectively. 
		We make a choice of setup data as follows.
		
		For each vertex $v$ of $\pfib(\Lambda)$,
		choose a triple
			$$(S'_v,\,\phi'_v,\,\{c'_\delta\}_{\delta|v}).$$
		Here $S'_v$ is an orientable finite cover of the
		JSJ subsurface $S_v$
		corresponding to $v$,
		and $\phi'_v$ is an automorphism of $S'_v$ fixing
		$\partial S'_v$, and for each end-of-edge $\delta$ adjacent to $v$,
		$c'_\delta$ is an elevation of
		the boundary component $c_\delta$ of $S_v$ 
		corresponding to $\delta$.		
		Implicitly, we also require that the mapping torus
			$$J'_v\,=\,\frac{S'_v\times[0,1]}{(x,0)\sim(\phi'_v(x),1)}$$
		is equipped with a finite covering projection to the JSJ piece $J$ of $M$ containing $S_v$,
		taking $S'_v\times\{0\}$ to $S_v$.
		
		\subsubsection{The construction}
		To construct a principal $\Rational^\times$-bundle over $\pfib(S)$,
		we start by taking a trivialized principal $\Rational^\times$-bundle over $S_v$:
			$$\mathscr{H}^+_v\cong S_v\times\Rational^\times$$
		for each vertex $v\in\pfib(\Lambda)$. 
		To glue these vertex principal $\Rational^\times$-bundles 
		$\mathscr{H}^+_v$ along the JSJ curves, 
		identify the fibers over $c_{\delta}$ and $c_{\bar{\delta}}$ via
		a principal-bundle equivalence for any end-of-edge $\delta$ of $\pfib(\Lambda)$, 
		where $\bar{\delta}$ denotes
		the opposite end of the same edge.
		Precisely, the gluing is given by a homeomorphism 
			$$c_{\delta}\times\Rational^\times\to c_{\bar{\delta}}\times\Rational^\times$$
		which naturally identifies the first factors and 
		multiplies the second factor by a positive rational ratio $s_\delta$,	satisfying
		\begin{equation}
			s_{\bar{\delta}}=s^{-1}_{\delta}.
		\end{equation}
					
		The ratios $s_\delta$ are determined as follows.			
		For each boundary curve $c_\delta$ of each JSJ subsurface $S_v$,
		the chosen elevated curve $c'_\delta$ 
		lies in a boundary component $T'_\delta$ of $J'_v$.
		Denote by $J_v$ and $T_\delta$ the JSJ piece of $M$ and the boundary component
		covered by $J'_v$ and $T'_\delta$, respectively.
		We associate to $c_\delta$ a nonzero integer
		\begin{equation}
			h_\delta\,=\,[T'_\delta:T_\delta]\,/\,[c'_\delta:c_\delta],
		\end{equation}
		where $[-:-]$ denotes the covering degree.
		The ratio $s_\delta$ is then defined to be
		\begin{equation}
			s_\delta\,=\,{h_{\bar{\delta}}}\,/\,{h_\delta}.
		\end{equation}
				
		The construction above, or essentially the ratios $s_\delta$,
		yields a principal $\Rational^\times$-bundle over the almost fiber part $\pfib(S)$,
		which we denote as $\mathscr{H}^+\to\pfib(S).$
		To correct the orientation character, we further twist $\mathscr{H}^+$ by taking
		the fiber product
				$$\mathscr{H}^+\times_{\pfib(S)}\tilde{S},$$
		where $\tilde{S}\to S$ is the orientable double cover corresponding to the kernel of 
		$\omega(S)\in H^1(S;\,\{\pm1\})\cong\mathrm{Hom}(\pi_1(S),\,\{\pm1\})$,
		regarded as a principal $\{\pm1\}$-bundle over $\pfib(S)$.
		The result is the associated principal $\Rational^\times$-bundle
			$$\mathscr{H}\to\pfib(S).$$
		
		\subsubsection{Verifications}
		In the above setting, by Definition \ref{spiralityCharacterDefinition},
		the spirality character of $\mathscr{H}$ can be calculated 
		using the following formula:
		
		\begin{formula}\label{spiralityH}
			For any directed $1$-cycle $\alpha$ dual to a cycle of
			directed edges $e_1,\cdots,e_n$,
				$$|\langle s(\mathscr{H}),\,[\alpha]\rangle|\,=\,\prod_{i=1}^{n}\,s_{\mathtt{ter}(e_i)}
				\,=\,\prod_{i=1}^{n}\,s_{\mathtt{ini}(e_i)}^{-1}
				\,=\,\prod_{i=1}^{n}\,\frac{h_{\mathtt{ini}(e_i)}}{h_{\mathtt{ter}(e_i)}},$$
			and the sign of $\langle s(\mathscr{H}),\,[\alpha]\rangle$ is 
			the value of the orientation character of $\pfib(S)$ on $[\alpha]$.
		\end{formula}
		
		We must verify that for any closed path $\alpha$ with basepoint in a JSJ curve of $\pfib(S)$, 
		$\langle s(\mathscr{H}),\,[\alpha]\rangle$ equals the dilatation rate
		$\lambda(\psi_\alpha)$, (Definition \ref{spiralityOfLoops}).
		Once this is done, we can infer immediately that the isomorphism class of 
		the principal $\Rational^\times$-bundle $\mathscr{H}$ 
		does not depend on the choice of the setup triple
		$(S'_v,\phi'_v,\,\{c'_\delta\}_{\delta|v})$, and 
		that the dilatation rate	$\lambda(\psi_\alpha)$ 
		does not depend on the choice of  the basepoint 
		of $\alpha$ either.

		Denote by $c$ the JSJ curve of $\pfib(S)$
		containing the basepoint of $\alpha$,
		and by $\alpha_1,\cdots,\alpha_n$ the subsegments
		in which $\alpha$ intersects with the JSJ subsurfaces of $\pfib(S)$,
		in cyclic consecutive order.
		Denote by $e_i$ the edges corresponding to the transition points
		between the segments $\alpha_i$ and $\alpha_{i+1}$.
		By convention, we understand the index $i$ as a residue class modulo $n$.
		Adopting the notations in the proof of Proposition \ref{psiAlpha},
		we have a restricted factorization of 
		$\psi_\alpha:A_c\dashrightarrow A_c$ into a composition of partial dilatations
		$$A_0\stackrel{\psi_1}{\dashrightarrow}A_1\stackrel{\psi_2}{\dashrightarrow}\cdots
		\stackrel{\psi_n}{\dashrightarrow} A_n,$$
		where $A_0$ and $A_n$ are both $A_c$. 
		
		In general, a dilatation rate can be assigned to a partial dilatation $\psi:A\dashrightarrow A'$ up to a sign,
		and the ambiguity can be resolved by specifying an orientation of $A$ and $A'$.
		To this end, we say \textit{ad hoc} that a finitely generated rank-$1$ abelian group $A$ is
		\emph{oriented} if a generator of
		the free cyclic quotient $\bar{A}$ is specified.
		With this generator identified to be $1$, 
		we may identify the rational vector space $A\otimes\Rational$
		with $\Rational$.	Suppose that $A$ and $A'$ are both oriented.
		With any defining elements $v\in A$ and $v'\in A'$,
		$\psi$ induces a	homomorphism of rational vector spaces
			$$\bar{\psi}:A\otimes\Rational\to A'\otimes\Rational,$$
		defined by $\bar{\psi}(v\otimes1)=v'\otimes1$.
		Since we have identified both $A\otimes\Rational$ and $A'\otimes\Rational$ with $\Rational$,
		the homomorphism $\bar{\psi}$ is the multiplication by a nonzero rational scalar,
		denoted as 
			$$\lambda(\psi)\in\Rational^\times,$$
		and we define $\lambda(\psi)$ to be the \emph{dilatation rate} of $\psi$.		
		It is easy to see that
		$\lambda(\psi)$ depends only on $\psi$ and the orientation
		of $A$ and $A'$.
		Moreover, the dilatation rate 
		is multiplicative under composition of partial dilatations.
		
		We fix an auxiliary choice of orientation for each $A_i$. Then
		\begin{equation}\label{productDilatationRate}
			\lambda(\psi_\alpha)=\lambda(\psi_1)\times\cdots\times\lambda(\psi_n).
		\end{equation}

		\begin{lemma}\label{lambdaPsiEye}
			For any  setup triple $(S'_v,\phi'_v,\,\{c'_\delta\}_{\delta|v})$,
			$$|\lambda(\psi_i)|\,=\,
				\frac{h_{\mathtt{ini}(e_i)}}{h_{\mathtt{ter}(e_{i-1})}}.$$
		\end{lemma}
		
		\begin{proof}
			We first assume that $\pfib(S)$ is orientable. Let $d$ be a positive integer
			divisible by all the covering degrees $[c'_\delta:c_\delta]$. For the mapping torus
			$J'_v$ of $\phi'_v$, consider the $d$-cyclic cover $J''_v$ of $J'_v$ dual to $S'_v$.
			Since $\phi'_v$ is boundary fixing, a closed leaf $b'_\delta$ 
			of the suspension flow on $T'_\delta$ intersects $c'_\delta$ in only one point,
			and the elevation $b''_\delta$ covers $b'_\delta$ of degree $d$.
			Identify the infinite cyclic cover of $T_\delta$ with $\tilde{T}_\delta$,
			the corresponding boundary component of the induced JSJ piece $\tilde{J}_v$
			of $\tilde{M}_S$, then $b''_{\mathtt{ini}(e_i)}$ can be regarded as an infinite-order element
			of $A_i$, namely, $b''_{\mathtt{ini}(e_i)}.y_i$, where $y_i$ is 
			the intersection of $\alpha$ with the JSJ curve $c_i$. Similarly,
			$b''_{\mathtt{ter}(e_{i-1})}$ can be regarded as an element of $A_{i-1}$.
			By definition, $\lambda(\psi_i)$ can be computed using
				$$\bar{\psi}_i(b''_{\mathtt{ter}(e_{i-1})}\otimes1)\,=\,b''_{\mathtt{ini}(e_i)}\otimes1,$$
			where  $b''_{\mathtt{ter}(e_{i-1})}\otimes1$ and $b''_{\mathtt{ini}(e_i)}\otimes1$
			are rational numbers under identification.
			Observe that the elevation $T''_\delta$ of $T'_\delta$ covers $T_\delta$, 
			and the covering map factors through 
			a characteristic finite cover of $T_\delta$ in which $c'_\delta$ has a lift,
			by the divisibility of $d$.
			This implies
			$$\frac{[T''_{\mathtt{ter}(e_{i-1})}:T_{\mathtt{ter}(e_{i-1})}]}{[c''_{\mathtt{ter}(e_{i-1})}:c_{\mathtt{ter}(e_{i-1})}] }
			\,=\,|b''_{\mathtt{ter}(e_{i-1})}\otimes 1|,$$
			%the right-hand side valued in $A_{i-1}\otimes\Rational$ isomorphic to $\Rational$,
			and
			$$\frac{[T''_{\mathtt{ter}(e_{i-1})}:T_{\mathtt{ter}(e_{i-1})}]}{[c''_{\mathtt{ter}(e_{i-1})}:c_{\mathtt{ter}(e_{i-1})}]}
			\,=\,d\cdot\frac{[T'_{\mathtt{ter}(e_{i-1})}:T_{\mathtt{ter}(e_{i-1})}]}{[c'_{\mathtt{ter}(e_{i-1})}:c_{\mathtt{ter}(e_{i-1})}]}.$$
			The same computation works for the initial endpoint of $e_i$ as well.
			%and 
			%$[T''_{\mathtt{ini}(e_i)}:T_{\mathtt{ini}(e_i)}]\,/\,[c''_{\mathtt{ini}(e_i)}:c_{\mathtt{ini}(e_i)}]$ 
			%equals the value of $b''_{\mathtt{ini}(e_{i})}$ in $A_{i}\otimes\Rational$.
			We obtain
			$$|\lambda(\psi_i)|\,=\,
					\frac{[T''_{\mathtt{ini}(e_i)}:T_{\mathtt{ini}(e_i)}]\,/\,[c''_{\mathtt{ini}(e_i)}:c_{\mathtt{ini}(e_i)}]}
					{[T''_{\mathtt{ter}(e_{i-1})}:T_{\mathtt{ter}(e_{i-1})}]\,/\,[c''_{\mathtt{ter}(e_{i-1})}:c_{\mathtt{ter}(e_{i-1})}]} 
					\,=\,\frac{h_{\mathtt{ini}(e_i)}}{h_{\mathtt{ter}(e_{i-1})}}.$$
					
			If $\pfib(S)$ is not orientable, in order to compute $\lambda(\psi_i)$,
			we need to work with the orientable double cover $S^*_v$ for any non-orientable $S_v$,
			and return to the $S_v$ level by only restricting to the JSJ cylinders. However,
			the idea is similar to the proof of \ref{eachPseudoDilatation} and the argument is 
			completely similar to the orientable case, so we omit the details. 
		\end{proof}
		
		By Formula \ref{spiralityH}, Lemma \ref{lambdaPsiEye} and Equation \eqref{productDilatationRate}, it follows
		that $\langle s(\mathscr{H}),[\alpha]\rangle$ equals $\lambda(\psi_\alpha)$ in the absolute value.
		To check the sign, observe that there are an odd number of negative values among all $\lambda(\psi_i)$
		if and only if the orientation character $\omega(\pfib(S))$ is $-1$ on $\alpha$. Note that $\tilde{M}_S$ is orientable
		and the JSJ cylinders are all two-sided. 
		Thus the sign matches by the fact that $\mathrm{sgn}(s(\mathscr{H}))$ equals $\omega(\pfib(S))$. 
		In particular, the quantity
			$$\langle s(\mathscr{H}),[\alpha]\rangle\,=\,\lambda(\psi_\alpha)$$
		is independent of the choice of the defining triple of $\mathscr{H}$ or the loop representative of $[\alpha]$.
		Thus $\mathscr{H}\to\pfib(S)$ is canonically defined up to isotopy of $\pfib(S)$ and bundle equivalence.
		
		To see the naturality of $\mathscr{H}\to\pfib(S)$, consider a finite covering map $\kappa:\tilde{S}\to S$.
		The almost fiber part $\pfib(\tilde{S})$ is the preimage of $\pfib(S)$.
		Choose a collection of setup data $(S'_v,\,\phi'_v,\,\{c'_\delta\}_{\delta|v})$
		for the canonical principal bundle $\mathscr{H}_S\to\pfib(S)$, possibly after modification by passing to multiples of $\phi'_v$, 
		we may choose a collection of setup data 
		$(\tilde{S}'_{\tilde{v}},\,\tilde{\phi}'_{\tilde{v}},\,\{\tilde{c}'_{\tilde{\delta}}\}_{\tilde{\delta}|\tilde{v}})$
		commutative with the covering map to construct the canonical pricipal bundle $\mathscr{H}_{\tilde{S}}\to\pfib(\tilde{S})$.
		Then Formula \ref{spiralityH} implies that
			$$s(\mathscr{H}_{\tilde{S}})\,=\,\kappa^*s(\mathscr{H}_S),$$
		so $\mathscr{H}_{\tilde{S}}$ is isomorphic to $\kappa^*\mathscr{H}_S$.
		This  completes the proof of Proposition \ref{bundleH}.

\section{Separability of surface subgroups}\label{Sec-separability}
	In this section, we prove a weak version of the aspirality criterion.
	Namely, we show that aspirality of the almost fiber part implies 
	a virtually embedded elevation. Invoking Przytycki--Wise \cite{PW-embedded},
	one will infer $\pi_1$-separability and the virtual	embedding property
	from aspirality of the almost fiber part.

	\begin{proposition}\label{aspiralityImpliesVirtuallyEmbeddedElevation}
		Let $M$ be a closed orientable aspherical $3$-manifold,  and $S$ be 
		a closed essentially	immersed subsurface.
		If	the almost fiber part $\pfib(S)$ is aspiral, 
		then $S$ has an embedded elevation in some finite cover of $M$.
	\end{proposition}
	
	The rest of this section is devoted to the proof of Proposition \ref{aspiralityImpliesVirtuallyEmbeddedElevation}.
	As the statement is about elevations of $S$, we may assume without loss of generality that $S$ is
	orientable.
	
	\subsection{Semicovers and merging}\label{Subsec-semicoverAndMerging}
		Finite covers of an orientable irreducible closed $3$-manifold 
		can be obtained by merging finite semicovers. 
		We briefly review some available techniques for our application.
		
		With respect to the JSJ decomposition of $M$,
		a semicover of $M$ is a $3$-manifold $N$ together with a closed immersion $N\to M$
		of which the restriction to any component of $\partial N$ covers a JSJ torus $M$,
		\cite[Definition 3.3]{PW-embedded}.  A \emph{finite} semicover is a semicover
		where $N$ is compact.

		The following lemma allows us to \emph{merge} atomic finite semicovers into a finite cover of a $3$-manifold. 
		More effecient operations such as
		\cite[Proposition 3.4]{PW-embedded} can be derived from this basic case.
		Given a positive integer $m$,
		we say that a finite semicover $N\to M$ is \emph{JSJ $m$-characteristic},
		if every elevation $\tilde{T}$ in $N$ of a JSJ torus $T$ in $M$
		is an $m$-characteristic cover of $T$,
    namely, such that every slope of $\tilde{T}$
    covers a slope of $T$ with degree $m$.

		\begin{lemma}[{Cf.~\cite[Proposition 4.2]{DLW}}]\label{mergeFiniteCovers}
            Let $M$ be an orientable irreducible closed $3$-manifold.
            Suppose $J'_1,\cdots,J'_s$
            are finite covers of all the JSJ pieces $J_1,\cdots,J_s$
            of $M$, respectively. Then there is a positive integer $m_0$, satisfying
            the following. For any positive integral multiple $m$ of $m_0$, there is
            a regular %(characteristic) 
            finite cover $\tilde{M}$ of $M$, which is JSJ $m$-characteristic,
            such that any elevation $\tilde{J}_i$ of a JSJ piece $J_i$
            is a cover of $J_i$ that factors through $J'_i$.
		\end{lemma}
	
	\subsection{Aspiral almost fiber part}
		The following lemma follows the same idea of Rubinstein and Wang \cite[Lemma 2.4]{RW}.
		
		\begin{lemma}\label{aspiralPhiS}
			As $\pfib(S)$ is orientable and aspiral, 
			there exists a finite semicover 
				$$Q(\pfib(S))\to M$$ 
			where $Q(\pfib(S))$ is a bundle over the circle,
			and the restricted immersion $\pfib(S)\looparrowright M$
			factors through $Q(\pfib(S))$ as a fiber.
		\end{lemma}
		
		\begin{proof}
			It suffices to construct a component $Q(F)$ of $Q(\pfib(S))$ for each component of $F$ of $\pfib(S)$. 
			We may assume that $F$ contains a JSJ curve $c$ in the interior, otherwise
			$F$ is already a virtual fiber of the carrier JSJ piece of $M$.
			
			Choose a basepoint $*$ of $S$ in $c$, which is mapped to a basepoint
			$x_0$ of $M$. Let $(\tilde{M}_S,\tilde{x}_0)$ be the pointed covering space of $M$ 
			corresponding to the image of $\pi_1(S,*)$ in $\pi_1(M,x_0)$. 
			We denote by $F_0$ the lift of $F$ at $\tilde{x}_0$, which is contained in a unique submanifold
			$\tilde{M}_F$ of $\tilde{M}_S$ bounded by the JSJ cylinders corresponding to the JSJ curves
			$\partial F_0$ of $S$. 
						
			We claim that there exists a lift $\tilde{F}_1$ of $F$ in $\tilde{M}_F$ other than $F_0$. 
			To see this, we take an auxiliary cell decomposition of $F$ with a base $0$-cell $*$. 
			As $\pfib(S)$ is aspiral, Proposition \ref{psiAlpha}
			implies that there exists an elevation $\tilde{x}_1$ other than $x_0$, 
			such that the $1$-skeleton of $F$ lifts
			to $\tilde{M}_F$ based at $\tilde{x}_1$.  In fact, with the notations 
			of Section \ref{Sec-theFirstRecurrenceMap},  if $\alpha_1,\cdots,\alpha_k$ 
			is a collection of cellular cycles of $F$ 
			based at $*$ which generate $\pi_1(F,*)$, 
			then each $\alpha_i$ induces a partial map $\phi_i:A_c\dashrightarrow A_c$
			by a partial dilatation, satisfying $\phi_i(mv_i)= mv_i$ for some $v_i\in A_c$ and all $m\in\Integral$.
			Possibly after passing to a common multiple, 
			we may assume that all the $v_i$ are the same $v\in A_c$, so the point $\tilde{x}_1$
			can be taken as $v.\tilde{x}_0$. There is no further obstruction so the lift of the $1$-skeleton of $F$
			can be extended to be a lift $F_1$ at $\tilde{x}_1$. 
			
			As $F_0$ is embedded, it follows that $F_1$ and $F_0$ are parallel disjoint embeddings
			by \cite{FHS}. There is a deck transformation $\tau$ of $\tilde{M}_F$ taking $F_0$ to $F_1$.
			The quotient of $\tilde{M}_F$ by the action of $\tau$ is a finite semicover $Q(F)$ of $M$, equipped 
			with the natural immersion $Q(F)\to M$ induced from the covering map $\tilde{M}_S\to M$.	
			It is clear that $Q(F)$ is a bundle with a fiber $F_0$ lifted from $F$, so $Q(F)$ is the component
			of $Q(\pfib(S))$ as desired. 
		\end{proof}
	
	\subsection{Partial suspensions}\label{Subsec-partialSuspensions}
		For any positive integer $m$,
		we consider an auxiliary space $Y_m(S)$,
		obtained by suspending the augumented almost fiber part of $S$ of degree $m$.
		The so-called partial suspension $Y_m(S)$ is an immersed subcomplex of $M$,
		extending the immersion of $S$. We will show that $Y_m(S)$ can be elevated to be
		embedded in some finite cover of $M$, as long as $m$ is an integral multiple
		of some positive integer $m_0$ depending on the immersion of $S$. 
		This will imply that $S$ is virtually embedded. 
						
		\subsubsection{Partial suspensions over the augumented almost fiber part}
		Since $\pfib(S)$ is orientable and aspiral, by Lemma \ref{aspiralPhiS}, we may take a semicover  
		$Q_1(\pfib(S))\to M$ through which $\pfib(S)$ factors as a fiber.
		Furthermore, for each JSJ curve of $S$ disjoint from $\pfib(S)$, we take a finite cover $T_1(c)$ of 
		JSJ torus $T(c)$ in which $c$ lifts to be embedded. The disjoint union 
		of $Q_1(\pfib(S))$ and all the $T_1(c)$, denoted as $Q_1(\pfib^*(S))$,
		is a fiber bundle over the circle, and
		any fiber is homeomorphic to the disjoint union of $\pfib(S)$ and all the $c$, denoted as $\pfib^*(S)$.
		There is an induced immersion
			$$Q_1(\pfib^*(S))\to M.$$
		We will refer to $\pfib^*(S)$ as the \emph{augumented almost fiber part} of $S$, and $Q_1(\pfib^*(S))$
		an \emph{augumented suspension} over $\pfib^*(S)$ of degree $1$.
				
		For any positive integer $m$, denote by $Q_m(\pfib^*(S))$ the $m$-cyclic cover of
		$Q(\pfib^*(S))$ dual to the fiber $\pfib^*(S)$. Denote by
			$$Y_m(S)=S\cup_{\pfib^*(S)}Q_m(\pfib^*(S))$$
		the space obtained identifying the almost fiber part $\pfib(S)$ of $S$ with a fiber of $Q_m(\pfib^*(S))$.
		The space $Y_m(S)$ is equipped with a locally embedding map:
			$$Y_m(S)\to M.$$
		We will refer to $Y_m(S)$ as a \emph{partial suspension} of $S$ over $\pfib^*(S)$ of degree $m$. 
		
		\subsubsection{Embedded elevations}
		By saying that a property about the immersion $j:S\looparrowright M$ holds for 
		any \emph{sufficiently divisible}	positive integer $m$, 
		we mean that there exists a positive integer $m_0$ depending on the immersion,
		and that the property holds for any positive integer $m$ divisible by $m_0$.
						
		\begin{lemma}\label{virtuallyEmbeddedPartialSuspensionElevation}
			For any sufficiently divisible positive integer $m$,
			the partial suspension $Y_m(S)$ has an embedded elevation in some finite cover of $M$.
		\end{lemma}
		
		\begin{proof}
			Observe that the partial suspension $Y_m(S)$ has an induced JSJ decomposition
				$$Y_m(S)\,=\,\bigcup_{\{Y_m(c_e)\}}\{Y_m(S_v)\}.$$
			Precisely, every JSJ curve $c_e$ of $S$ corresponds to
			an induced JSJ torus $Y_m(c_e)$,
			the elevation in $Q_m(\pfib^*(S))$ of the JSJ torus of $M$ containing $c_e$.
			Every JSJ surface $S_v$ of $S$ corresponds to an induced JSJ piece $Y_m(S_v)$:
			if $S_v$ is contained in $\pfib(S)$, $Y_m(S_v)$ is the corresponding component
			of $Q_m(\pfib(S))$; if $S_v$ is not contained in $\pfib(S)$, $Y_m(S_v)$ is the union of $S_v$ and
			the JSJ tori $Y_m(c_\delta)$ for each boundary component $c_\delta$ of $\partial S_v$.
			The point of the induced decomposition
			is that the JSJ pieces corresponding to a non-virtual-fiber JSJ subsurface should not be the subsurface only, 
			but with the hanging adjacent JSJ tori, because these tori are $\pi_1$-visible in 
			the carrier JSJ piece of $M$ as well. If one thickens up the immersed $Y_m(S)$ in $M$
			to be an immersed compact orientable $3$-manifold $\mathcal{Y}_m(S)$, the induced
			JSJ decomposition of $\mathcal{Y}_m(S)$ in the usual sense will naturally give rise to 
			the induced JSJ decomposition of $Y_m(S)$ as described above.
			
			\medskip\noindent\textbf{Step 1}. For any JSJ piece  $Y_m(S_v)$,
			carried by a JSJ piece $J$ of $M$,
			we claim that for any sufficiently large positive integer $m$,
			there exists a regular finite cover $J'$ of $J$ in which
			any elevation of $Y_m(S_v)$ is embedded.
			
			If $S_v$ is contained in $\pfib(S)$, there is nothing to prove, as $Y_m(S_v)$ is already
			a finite cover of $J$. If $S_v$ is not contained in $\pfib(S)$, there are two cases:
			either $S_v$ is a properly immersed vertical annulus in a nonelementary 
			Seifert fibered piece $J$, or $S_v$ is a properly immersed geometrically finite subsurface
			in a hyperbolic piece $J$.
			
			When $S_v$ is a properly immersed geometrically finite subsurface of a hyperoblic piece $J$,
			by \cite[Theorem 4.1]{PW-embedded}, $Y_m(S_v)$ is $\pi_1$-injective
      and relatively quasiconvex if $m$ is sufficiently large. Moreover, in this case,
      it is a consequence of the relative quasiconvex separability due to
      Wise \cite[Theorem 16.23]{Wise-long}
      (cf.~\cite[Corollary 4.2]{PW-embedded}) that $\pi_1(Y_m(S_v))$ is indeed separable
			in $\pi_1(J)$. Thus, there is a regular finite cover $J'$ of $J$ in which any elevation of
			$Y_m(S_v)$ is embedded.
			
			When $S_v$ is a properly immersed vertical annulus in a nonelementary Seifert fibered piece $J$,
			the base $2$-orbifold $\mathcal{O}$ is of the hyperbolic type. The JSJ piece $Y_m(S_v)$ is
			a circle bundle over a graph $Z_m(a_v)$, where $a_v$ is a properly immersed arc in $\mathcal{O}$,
			and at each endpoint of $a_v$ there is a wedged circle covering the corresponding
			boundary component of $\mathcal{O}$ of some degree divisible by $m$. A direct argument 
			in hyperbolic geometry, or a $2$-dimensional version of the argument in the previous case,
			implies that $\mathcal{O}$ has a regular finite cover $\mathcal{O}'$ in which any elevation
			of $Z_m(a_v)$ is embedded. The base cover $\mathcal{O}'$ of $\mathcal{O}$
			induces a regular finite cover $J'$ of $J$, in which any elevation of $Y_m(S_v)$ is embedded.  
			
			This proves the claim of Step 1.
			
			\medskip\noindent\textbf{Step 2}. 
			We claim that for any sufficiently divisible integer $m$,
			there exists a regular finite cover $M'$ of $M$ in which all the elevations
			of the JSJ pieces of $Y_m(S)$ are embedded. 
			
			We apply the merge trick. We may assume that $m$ is sufficiently large so that the claim
			of Step 1 holds for all the JSJ pieces $Y_m(S_v)$, giving rise to regular finite covers $J'_v$ for each carrier JSJ piece 
			of $M$. If some of these $J'_v$ cover the same JSJ piece $J$ of $M$, we pass to a common regular finite cover $J'$ of
			these $J'_v$; if some JSJ piece $J$ of $M$ is not covered by any of these $J'_v$, we choose $J'$ to be $J$ itself.
			By Lemma \ref{mergeFiniteCovers}, there exists a regular finite cover $M'$ of $M$ merging all these covers $J'$ of $J$.
			Note the conclusion in the claim of Step 1 is not affected passing to further regular finite cover $J''$ of $J$ that factors through $J'$.
			Thus, in the regular cover $M'$ of $M$, all the elevations of the JSJ pieces of $Y_m(S)$ are embedded, as claimed.
			
			\medskip\noindent\textbf{Step 3}.
			We finish the proof by showing that 
			the regular finite cover $M'$ of $M$ given by Step 2 
			has a further finite cover $\tilde{M}$ in which
			$Y_m(S)$ has an embedded elevation. Indeed, $\tilde{M}$ will be induced
			from a finite cover of the dual graph of $M'$.
			
			In fact, if $M'$ is a regular finite cover of $M$ given by Step 2, any elevation
				$$Y'_m(S)\to M'$$
			induces a map between the graphs dual to the decomposition
				$$\Lambda(Y'_m(S))\to \Lambda(M').$$
			Because the edge spaces of $Y'_m(S)$ are all finite covers of JSJ tori of $M'$,
			and because each vertex space of $Y'_m(S)$ is an elevated JSJ piece embedded 
			in a JSJ piece of $M'$, the induced map between the dual graphs is a combinatorial 
			local isometry, which is, in particular, $\pi_1$-injective.
			Because $\pi_1(\Lambda(M'))$ is a free group, and hence is LERF by Scott \cite{Sc-LERF},
			$\Lambda(M')$ has a regular finite cover in which 
			any elevation $\Lambda(Y'_m(S))$ is an embedded subgraph.
			Therefore, in the induced regular finite cover $\tilde{M}$ of $M$,
			any elevation of $Y_m(S)$ is embedded. 
			
			This completes the proof. 
		\end{proof}
		
		Because the augumented partial suspension $Y_m(S)$ contains a $\pi_1$-injective
		embedded copy of $S$, Lemma \ref{virtuallyEmbeddedPartialSuspensionElevation}
		implies that $S$ has an embedded elevation in a finite cover of $M$.
		This completes the proof of Proposition \ref{aspiralityImpliesVirtuallyEmbeddedElevation}.
		
\section{Virtual tautness}\label{Sec-virtualTautness}
	In this section, we prove virtual tautness of essentially embedded orientable subsurfaces.
	Then we  summarize the proof of Theorem \ref{main-aspiralityCriterion}
	in the end of this section.
	
	Recall that the complexity $\chi_-(S)$ of a (possibly disconnected) compact orientable surface $S$
	is defined to be the opposite of the sum of the negative Euler characteristics 
	of its components. In \cite{Th-norm}, Thurston introduces a semi-norm
	on $H_2(M,\partial M;\,\Rational)$ for any compact $3$-manifold $M$, 
	known as the \emph{Thurston norm}.
	The Thurston norm $x(\xi)$ of an integral relative second homology class $\xi$
	is defined to be the minimum of $\chi_-(S)$ among all
	compact oriented properly embedded subsurfaces $S$ of $M$
	representing $\xi$.
	We say that an oriented (possibly disconnected) compact properly embedded subsurface
	of a compact $3$-manifold is \emph{taut} if it realizes the Thurston
	norm of its homology class.
	
	\begin{proposition}\label{essentiallyEmbeddedImpliesVirtuallyTaut}
		Let $M$ be an orientable closed aspherical $3$-manifold, and 
		$S$ be an oriented closed essentially embedded subsurface of $M$.
		Then $S$ is virtually taut.
	\end{proposition}
	
	The proof of Proposition \ref{essentiallyEmbeddedImpliesVirtuallyTaut}
	follows the idea in the proof of Proposition \ref{aspiralityImpliesVirtuallyEmbeddedElevation},
	with strengthened arguments	at a number of places.
	
	\subsection{Completing a semicover}
		We first need a stronger version of Przytycki--Wise \cite[Proposition 3.4]{PW-embedded},
		so as to complete a connected semicover without passing to a further finite cover.
				
		\begin{lemma}\label{completingASemicover}
			If $N$ is a connected finite semicover of an orientable closed aspherical $3$-manifold $M$,
			then $N$ has an embedded lift in a finite cover of $M$.
			In fact, the semi-covering $N\to M$ is $\pi_1$-injective and $\pi_1$-separable.
		\end{lemma}
		
		\begin{proof}
			After choosing auxiliary basepoints, we prove that $\pi_1(N)$ injects
			to be a separable subgroup of $\pi_1(M)$.
			Then $N$ is virtually embedded by Scott \cite[Lemma 1.4]{Sc-LERF}.
			Because $N$ has an elevation embedded in a finite cover of $M$,
			\cite[Proposition 3.4]{PW-embedded},
			and separability is a virtual property, 
			we may passing the pair and assume $N$ is embedded	in $M$.			
			Choose a basepoint $*$ of $M$ contained in $N$. Below we abuse the notation
			of pointed loops and elements in fundamental groups. 
			
			For any element $\alpha\in\pi_1(M)$	not contained in $\pi_1(N)$,
			$\alpha$ is pointed homotopic to a concatenation $\tau\gamma\bar{\tau}$,
			where $\gamma$ is a loop freely homotopic to $\alpha$,
			and $\tau$ is a path from $*$ to a point on $\gamma$. 
			Moreover, we require the $1$-complex $\gamma\cup\tau$ 
			to be transverse to the JSJ tori, minimizing the number of intersections
			subject to the conditions above. Because each JSJ torus
			is $\pi_1$-separable in an adjacent JSJ piece \cite[Theorem 1]{LoN}, 
			for every embedded subpath $\gamma\cup\tau$ which is properly embedded
			in a JSJ piece of $M$, there is a regular finite cover of the JSJ piece so that 
			any elevation of the path connects two distinct boundary components.
			Applying merging (Lemma \ref{mergeFiniteCovers}),
			we obtain a regular finite pointed cover $M'$ of $M$ 
			with a pointed elevation $N'$ of $N$. Note that $\pi_1(N')$ equals $\pi_1(N)\cap\pi_1(M')$.
			For any power $\alpha'$ of $\alpha$ contained in $\pi_1(M')$,  
			$\alpha'$ takes the form $\tau'\gamma'\bar{\tau}'$ up to pointed homotopy,
			where $\gamma'$ is a finite cyclic cover of $\gamma$, and $\tau'$ is the pointed lift of $\tau$.
			The construction of $M'$ implies that the projection of the pointed $1$-complex 
			$\gamma'\cup\tau'$ to pointed JSJ dual graph $\Lambda_{M'}$ of $M'$
			is a local embedding. 
			Moreover, since $\alpha\not\in\pi_1(N)$ and hence $\alpha'\not\in\pi_1(N')$, 
			the projection of $\gamma'\cup\tau'$	cannot be contained in the subgraph $\Lambda_{N'}$ dual
			to $N'$.	Using the fact that $\Lambda_{N'}$ is $\pi_1$-separable 
			in $\Lambda_{M'}$ \cite[Theorem 2.1]{Sc-LERF}, 
			we obtain a finite cover $M''$ of $M'$ induced by a cover of $\Lambda_{M'}$,
			such that $\pi_1(N')$ is contained in $\pi_1(M'')$ but
			$\alpha'\not\in\pi_1(M'')$, or in other words, $\alpha'$ is separable
			from $\pi_1(N')$.
			We conclude that $\alpha$ is separable from $\pi_1(N)$ 
			in $\pi_1(M)$.			
		\end{proof}

	\subsection{Criteria for tautness}
		Next, we need to introduce some controllable conditions
		to guarantee tautness of $S$ in terms of its induced JSJ decomposition.
	
		\begin{lemma}\label{globalTautness}
			If each JSJ subsurface of $S$ is taut and each JSJ piece of $M$ contains
			at most one JSJ subsurface of $S$, then $S$ is taut.
		\end{lemma}
			
		\begin{proof}
			Let $S'$ be a taut subsurface representing $[S]\in H_2(J;\,\Rational)$, then 
			$S'$ can be isotoped to intersect the JSJ tori of $M$ in directly parallel essential curves.
			Note that for each JSJ piece $J$, $[S\cap J]$ and $[S'\cap J]$ 
			represents the same homology class in $H_2(J,\partial J)$ by excision.
			Since the JSJ subsurfaces $S_v$ of $S$ are taut and contained by distinct JSJ pieces of $M$,
				$$x([S])=\chi_-(S')=\sum_J\chi_-(S'\cap J)\geq \sum_v\chi_-(S_v)=\chi_-(S)\geq x([S]).$$
			Hence $x([S])$ equals $\chi_-(S)$, which means $S$ is taut.
		\end{proof}
		
		\begin{lemma}\label{localTautness}
			If a JSJ subsurface $S_v$ contained in a JSJ piece $J$ is either
			a fiber or a retract of $J$, then $S_v$ is taut in $J$.
		\end{lemma}
		
		\begin{proof}
			If $S_v$ is a fiber of $J$, it is taut because it is a compact leaf of 
			the taut codimension-$1$ foliation of $J$ given by the bundle structure.
			
			If $S_v$ is a retract of $J$, consider any other taut subsurface $E$ representing $[S_v]$
			in $H_2(J,\partial J;\Rational)$. Possibly after gluing indirectly parallel pairs of boundary components,
			which does not affect the complexity, we may assume for each boundary component $T$ of $J$
			that $S_v\cap T$ and $E\cap T$ are directly parallel curves of the same multiplicity, possibly zero.
			The existence of the retraction $J\to S_v$ implies that $S_v\cap T$ must either be connected or empty.
			Moreover, the retraction induces a proper degree one map $E\to S_v$, so
			$$x([S_v])=\chi_-(E)\geq\chi_-(S_v)\geq x([S_v]).$$
			Hence $S_v$ is taut in $J$.			
		\end{proof}
	
	\subsection{Virtual embedding of partial suspensions}
		We modify some argument in the construction of partial suspensions
		in Subsection \ref{Subsec-partialSuspensions} to fit the assumptions
		of Lemmas \ref{globalTautness}, \ref{localTautness}.
		Adopt the notations there.
							
		\begin{lemma}\label{localPartialSuspension}
			Let $S_v$ be a JSJ subsurface of $S$ carried by a JSJ piece of $M$.
			For any sufficiently divisible positive integer $m$, the following conditions are satisfied.
			\begin{enumerate}
				\item If $S_v$ is contained in $\pfib(S)$, then $Y_m(S_v)$ is a finite cover $\tilde{J}$ of $J$,
				and $S_v$ intersects each boundary component of $Y_m(S_v)$ in a curve;
				\item If $S_v$ is not contained in $\pfib(S)$, then some finite cover $\tilde{J}$ 
				of $J$ contains an embedded lift of $Y_m(S_v)$, and hence $S_v$, as a retract.
			\end{enumerate}
		\end{lemma}
		
		\begin{proof}
			If $S_v$ is contained in $\pfib(S)$, $Y_1(S_v)$, which is $Q_1(S_v)$, can be taken to be $J$.
			Since $S_v$ is oriented and embedded, $Y_m(S_v)$ is the $m$-cyclic cover dual to $S_v$
			of $J$, so for any sufficiently divisible positive integer, $S_v$ intersects each boundary of $Y$
			in a curve.
			
			If $S_v$ is not contained in $\pfib(S)$, 
			$S_v$ is either geometrically finite or vertical.
			In the former case, we strengthen the argument before
			by invoking Wise \cite[Theorem 16.23]{Wise-long},
			and conclude that $\pi_1(Y_m(S_v))$ is a virtual retract of $\pi_1(J)$.
			This implies a finite cover $\tilde{J}$ in which $Y_m(S_v)$ lifts, and
			the $\pi_1$-retraction map $\tilde{J}\to Y_m(S_v)$ can be constructed
			by extending the identity map restricted to the lift of $Y_m(S_v)$.
			In the latter case, we may similarly strengthen the argument before
			to conclude that $Y_m(S_v)$ is a virtual retract of $J$, 
			so we omit the details here. Note that $S_v$ is also a virtual retract of $J$
			as it is a retract of $Y_m(S_v)$.
		\end{proof}

		To prove Proposition \ref{essentiallyEmbeddedImpliesVirtuallyTaut},
		we take a finite cover $\tilde{J}_v$ of the JSJ piece of $M$ carrying $S_v$,
		for some sufficiently divisible positive integer $m$ 
		satisfying the conclusion of Lemma \ref{localPartialSuspension},
		The finite covers $\tilde{J}_v$ contain embedded lifts of $Y_m(S_v)$,
		so we glue these $\tilde{J}_v$ up along some
		boundary components in the same way we assemble the augumented partial suspension
		$Y_m(S)$ from $Y_m(S_v)$. 
		The result is a connected finite semicover $N$ of $M$ in which
		$Y_m(S)$ lifts. By Lemma \ref{completingASemicover}, we complete $N$ as 
		a finite cover $\tilde{M}$ of $M$. The finite cover contains a lift of $Y_m(S)$, and
		hence a lift of $S$ satisfying the conditions of the tautness criteria Lemmas \ref{globalTautness}, \ref{localTautness}.
		It follows that $S$ lifts to be a taut subsurface of $\tilde{M}$.
		
		This completes the proof of Proposition \ref{essentiallyEmbeddedImpliesVirtuallyTaut}.

	\subsection{The proof of Theorem \ref{main-aspiralityCriterion}}
		Let $S$ be a closed essentially immersed subsurface of a closed orientable aspherical $3$-manifold $M$.
					
		(1) $\Rightarrow$ (2):
		If the almost fiber part $\pfib(S)$ of $S$
		is aspiral, by Proposition \ref{aspiralityImpliesVirtuallyEmbeddedElevation}, $S$ has a virtually embedded elevation. The latter implies that $S$
		is $\pi_1$-separable by Przytycki--Wise \cite{PW-embedded}, or that $S$ is virtually embedded
		\cite[Lemma 1.4]{Sc-LERF}.  
		
		(2) $\Rightarrow$ (3): 
		If $S$ is virtually embedded, we may pass to a finite cover and assume 
		that $S$ is already embedded.
		First suppose that $S$ is orientable. By Proposition \ref{essentiallyEmbeddedImpliesVirtuallyTaut},
		$S$ is taut in some finite cover $\tilde{M}$ of $M$. In this case, Gabai \cite{Ga} constructs a taut foliation $\mathscr{F}$
		of $\tilde{M}$ that contains $S$ as a leaf. In fact, $\mathscr{F}$ can be either finite depth or smooth.
		
		When $S$ is not orientable, we derive a virtual taut foliation from the orientable
		case as follows. The boundary of a compact regular neighborhood $\mathcal{N}$
		of $S$ in $M$ is a twisted interval bundle over $S$ whose boundary is
		isomorphic to the characteristic orientable double cover $S^*$ of $S$, equipped with an orientation-reversing 
		free involution $\nu:S^*\to S^*$ which is the deck transformation.
		As $S^*$ is essentially embedded and orientable, the previous case yields
		a finite cover $M'$ in which $S^*$ is a leaf of a taut foliation $\mathscr{F}^*$.
		Let $M'(S)$ be the compact $3$-manifold obtained from $M'$ by removing a regular
		neighborhood of $S^*$, so $\partial M'(S)$ is homeomorphic to $(+S^*)\sqcup (-S^*)$,
		two oppositely oriented copies of $S^*$. 
		Glue $+S^*$ to itself by $\nu$ and similarly for $-S^*$. 
		The result is a $3$-manifold $\tilde{M}$ which finitely covers $M$ in an obvious way.
		Moreover, the quotient of $\pm S^*$ are two copies
		of $S$ in $\tilde{M}$, and the complement can be foliated by the leaves of $\mathscr{F}^*$.
		The result is clearly a taut foliation $\mathscr{F}$ of $\tilde{M}$,
		which contains a lift of $S$ as a leaf,
		so we have proved the non-orientable case.
		
		(3) $\Rightarrow$ (1):
		If $S$ is a compact leaf of a taut foliation of $M$. Possibly after passing to a finite cover of $M$
		and an elevation of $S$, we may assume without loss of generality that $S$ is orientable,
		by the naturality of spirality (Proposition \ref{bundleH}).
		Since $[S]\in H_2(M;\Integral)$ is nontrivial,
		possibly after passing to a sufficiently divisible finite cyclic cover of $M$ dual to $S$, 
		we may assume that each JSJ curve
		of $S$ lies on a distinct JSJ torus of $M$. Then Formula \ref{spiralityH} immediately 
		implies that the spirality of $\pfib(S)$ is trivial. Therefore, $S$ is aspiral in the almost fiber part.
		
		This completes the proof of Theorem \ref{main-aspiralityCriterion}.

\section{Immersed subsurfaces transverse to suspension flow}\label{Sec-flowTransverse}
	In this section, we derive a formula to calculate  the spirality for immersed subsurfaces transverse to
	suspension flows, which generalizes the formula of Rubinstein--Wang \cite{RW} in the graph manifold case.
	We apply the formula to prove Corollary \ref{separableExamples} in the end of this section.
	
	\begin{definition}\label{pseudoGraphManifold}
		An orientable closed aspherical $3$-manifold is said to be a \emph{pseudo graph manifold}, 
		if every JSJ piece is enriched with a Nielsen--Thurston supsension flow. In other words,
		each JSJ piece is homeomorphic to 
		the mapping torus of an automorphism of a compact orientable surface which is either periodic or pseudo-Anosov,
		and the flow structure of the JSJ piece is the pullback of the suspension flow.
		The \emph{degeneracy slope} of a JSJ piece of $M$ is a closed leaf on a boundary component, which is unique
		up to isotopy.
	\end{definition}
	
	Let $c$ be a transversely sided immersed loop carried by a JSJ torus $T$ of $M$, namely,
	$c$ is immersed in $T$ and assigned with a side of $T$. Denote by  $l^\pm$ the degeneracy slopes
	of $T$ induced from the $\pm$ side of $c$. We define
	\begin{equation}
		\sigma_M(c)\,=\,\frac{\gin(c,l^-)}{\gin(c,l^+)}
	\end{equation}
	where $\gin(c,l^\pm)$ denotes the (minimal) geometric intersection number between $c$ and $l^\pm$.
	
	On the other hand, let $\gamma$ be a directed path in a JSJ piece $J$ of $M$ joining two JSJ torus
	$T_{\mathtt{ini}(\gamma)}$ and $T_{\mathtt{ter}(\gamma)}$ on the boundary. Denote by $l_{\mathtt{ini}(\gamma)}$ and $l_{\mathtt{ter}(\gamma)}$
	the degeneracy slopes respectively. We define
	\begin{equation}
		\rho_M(\gamma)\,=\,\frac{\mathrm{\ell}(l_{\mathtt{ini}(\gamma)})}{\mathrm{\ell}(l_{\mathtt{ter}(\gamma)})}
	\end{equation}
	where $\ell$ denotes the length measured in the flowing time unit.
	Note that $\rho_M(\gamma)$ does not change under rescale of flow.
	
	\begin{formula}\label{formulaRW}
		Let $M$ be an orientable closed aspherical pseudo graph manifold, and $S$ be an
		orientable closed essentially immersed subsurface.
		Suppose that the almost fiber part $\pfib(S)$ is transverse to the JSJ tori and the Nielsen--Thurston flow
		on the JSJ pieces of $M$. For any immersed directed loop $\alpha$ in $\pfib(S)$ transversely intersecting
		the (possibly repeating) JSJ curves $c_1,\cdots,c_n$, denote by
		$\alpha_i$ the subpath joining $c_{i-1}$ and $c_i$ with the index understood modulo $n$,
		and consider $c_i$ to be transversely sided in the direction of $\alpha$.
		The spirality of $\alpha$ is a product:
				$$s(\alpha)\,=\,\sigma_M(\alpha)\,\rho_M(\alpha),$$
		where
				$$\sigma_M(\alpha)\,=\,\sigma_M(c_1)\times\cdots\times \sigma_M(c_n),$$
		and 
				$$\rho_M(\alpha)\,=\,\rho_M(\alpha_1)\times\cdots\times \rho_M(\alpha_n).$$
	\end{formula}
	
	For example, if $M$ is a graph manifold and $S$ is horizontally immersed, we recover
	the original formula of Rubinstein--Wang \cite{RW}. Perhaps the most similar situation
	for a general $M$ is when $S$ is an almost fiber homologically lying in a fibered cone,
	and is transverse to a canonically induced suspension flow.
	
	%The correction factor $\rho_M(\alpha)$
	
	\begin{lemma}\label{equiperiodic}
		If every hyperbolic JSJ piece of $M$
		is the mapping torus of a pseudo-Anosov automorphism of a compact orientable surface 
		which fixes all the periodic points on the boundary,
		then:
			$$\rho_M(\alpha)\,=\,1.$$
	\end{lemma}
	
	\begin{proof}
		In this case, the length of degeneracy slopes on the boundary of a hyperbolic JSJ piece
		depends only on the piece. On the other hand, the degeneracy slopes 
		on the boundary of a Seifert fibered piece is necessarily an ordinary fiber,
		so their length depends only on the piece as well.
		It follows that each $\rho_M(\alpha_i)$ equals $1$ so $\rho_M(\alpha)$ equals $1$.		
	\end{proof}
	
	\begin{remark}
		In general, the correction factor $\rho_M$ can be nontrivial 
		since orbit length of periodic points on the boundary 
		for pseudo-Anosov maps may differ as the component varies.
		For explicit examples, see the pseudo-Anosov braids $\beta_{m,n}$ in \cite{HK}.
	\end{remark}

		In order to derive the Formula \ref{formulaRW} from Formula \ref{spiralityH}, 
		we take a setup triple $(S'_v,\phi'_v,\{c'_\delta\}_{\delta|v})$ as follows.
		For each JSJ subsurface $S_v$ of $\pfib(S)$, take
		$S'_v$ to be $S_v$, and hence $c'_\delta$ to be $c_\delta$.
		Denote by $J_v$ the JSJ piece of $M$ carrying $S_v$, and
		take $J'_v$ to be a finite cover of $J_v$ with a fiber $S'_v$
		and a boundary-fixing monodromy $\phi'_v$.
		Note that $\phi'_v$ can be freely isotopic to the Nielsen--Thurston normal form
		which is either periodic or pseudo-Anosov, 
		unique up to conjugacy.
		Possibly after passing to a power of $\phi'_v$,
		we require all the periodic points 
		of the Nielsen--Thurston normal form of $\phi'_v$ 
		on $\partial S'_v$ to be fixed by $\phi'_v$.
				
		The assumption that $\pfib(S)$ is flow-transversely immersed
		guarantees that the suspension flow of 
		the Nielsen--Thurston normal form of $\phi'_v$
		on $J'_v$ is conjugate to the pullback of
		the flow of $J_v$. In fact, 
		if $S_v$ is carried by a Seifert-fibered piece, $S_v$ is horizontally immersed,
		and the flows are both conjugate to the Seifert fibration of $J'_v$.
		If $J_v$ is atoroidal, the result is due to Fried \cite[Theorem 7]{Fried}, cf.~McMullen \cite[Theorem 3.1]{McMullen}.
		In particular, any degeneracy slope $l'_\delta$ of $J'_v$ covers
		the degeneracy slope $l_\delta$ of $J_v$ up to homotopy.
		%Furthermore, the assumption that the Nielsen--Thurston normal form
		%of $\phi'_v$ fixes all the periodic points on the boundary
		%implies that the length $\ell(l'_\delta)$ depends only on the carrier vertex $v$,
		%so we write it as $\ell'_v$.
		%Lemma \ref{equiperiodic} implies that all the degeneracy slopes $l_\delta$ of $J_v$
		%have equal length measured in the flowing time unit, and the same holds for the degeneracy slopes
		%$l'_\delta$ of $J'_v$. 
		Therefore, the covering degrees between degeneracy slopes can be computed by:
			$$[l'_\delta:l_\delta]\,=\,\ell(l'_\delta)\,/\,\ell(l_\delta).$$
				
		For any boundary component $c_\delta$ of $S_v$ carried by a boundary component $T_\delta$ of $J_v$,
		$c'_\delta$ is a lift. We have
			$$[c'_\delta:c_\delta]\,=\,1,$$
		and
			$$[T'_\delta:T_\delta]\,=\,[l'_\delta:l_\delta]\cdot[c'_\delta:c_\delta]\cdot\frac{\gin(c_\delta,l_\delta)}{\gin(c'_\delta,l'_\delta)}
			\,=\,\frac{\gin(c_\delta,l_\delta)\,/\,\ell(l_\delta)}{\gin(c'_\delta,l'_\delta)\,/\,\ell(l'_\delta)}.$$			
		Observe that
		$\gin(c'_\delta,l'_\delta)$ equals $1$ on $T'_\delta$ and $\ell(l'_\delta)$
		depends only on the carrying vertex $v$, by our assumption
		that the Nielsen--Thurston normal form of $\phi'_v$ fixes all periodic points 
		on the boundary. Thus the factors on the denominator are cyclically cancelled 
		when plugged into Formula \ref{spiralityH}.
		Now Formula \ref{formulaRW} can be deduced Formula \ref{spiralityH}
		by simple substitution.

	\begin{proof}[{Proof of Corollary \ref{separableExamples}}]
		Let $\theta:F\to F$ be an orientation-preserving homeomorphism of
		an oriented closed surface $F$ of negative Euler characteristic.
		Let $S$ be an oriented closed essentially immersed subsurface of the mapping torus $M_\theta$ 
		transverse to the JSJ tori and to the Nielsen--Thurston suspension flow supported in the JSJ pieces.

		Suppose that the monodromy $\theta$ has no nontrivial fractional Dehn twist coefficient. 
		It suffices to show that some elevation of $S$ in a finite cover of $M_\theta$
		is aspiral in the almost fiber part, by Theorem \ref{main-aspiralityCriterion}.
		By passing to a power of $\theta$, we may assume that $\theta$ fixes all the
		periodic points on the boundary of the pseudo-Anosov part.
		For any loop $\alpha$ in the almost fiber part of $S$,
		we have $\rho_{M_\theta}(\alpha)$ equals $1$ by Lemma \ref{equiperiodic}.
		Because $\theta$ has no nontrivial fractional Dehn twist coefficient,
		for all JSJ tori of $M_\theta$, the degeneracy slopes induced from both sides coincide with each other.
		Hence $\sigma_{M_\theta}(\alpha)$ equals $1$.
		By Formula \ref{formulaRW} we conclude that $S$ is aspiral in the almost fiber part.
	\end{proof}

\section{Application}\label{Sec-application}
	In this section, we exhibit new examples of subsurfaces
	essentially immersed but not virtually embedded,
	as asserted by Corollary \ref{nonSeparableExamples}.
	We	construct the subsurface by applying techniques 
	developed in \cite{PW-graph,PW-mixed,DLW},
	and detect a loop with nontrivial spirality 
	using Formula \ref{formulaRW}.
	
	\begin{example}
		Let $M_\theta$ be a closed orientable surface bundle fibering over the circle
		with a fiber $F$ and a monodromy $\theta:F\to F$. Suppose that $\theta$ has a nontrivial
		fractional Dehn twist coefficient along a reduction curve $e$. We construct
		a closed essentially immersed subsurface $S$ of $M_\theta$ 
		which is not virtually embedded.
	\end{example}
	
	The idea is first to construct locally an essentially immersed almost fiber subsurface $E$, 
		within the JSJ pieces adjacent to the carrier JSJ torus $e$.
		We carefully make sure that $E$ has nontrivial spirality using the fact that
		the fractional Dehn twist coefficient of $\theta$ along $e$ is nontrivial. 
		Take two oppositely oriented copies
		$E$ and $\bar{E}$. Then, we can merge them along boundary
		with other virtually embedded subsurfaces to obtain a closed subsurface of the $3$-manifold. 
		Below we give an outline of a construction.
	
	\begin{proof}[Construction]
		Rewrite $M_\theta$ as $M$ for simplicity, and denote by $T$ the JSJ torus carrying
		$e$.
		Possibly after passing to a finite cover of $M$ and an elevation of $F$, 
		we may assume that the $\theta$ preserves all the reduction curves 
		componentwise and the periodic or pseudo-Anosov
		normal forms in the rest regions fix all periodic points on the boundary.
		We may also assume that the JSJ torus $T$ is adjacent two distinct JSJ pieces $J^\pm$.
		Denote by $l^\pm$ the degeneracy slopes
		on $T$ induced from $J^\pm$, directed in the flow direction,
		and by $F^\pm$ the fiber $F\cap J^\pm$.
		We have
			$$\gin(e,l^\pm)\,=\,1,$$ 
		and in $H_1(T;\Integral)$, choosing a suitable direction of $e$,
			$$[l^+]-[l^-]\,=\,k\cdot[e],$$
		where $k$ is the fractional Dehn twist coefficient of $\theta$, 
		a nonzero integer under our assumption of simplification.
		
		Take a regular JSJ characteristic finite cover of $J^\pm$, 
		denoted by $\tilde{J}^\pm$,
		such that $\tilde{J}^\pm$ has at least three boundary components
		covering $T$,
		(cf.~\cite{CLR-bounded} and Subsection \ref{Subsec-semicoverAndMerging}).		
		Then there exists two boundary components $\tilde{T}^\pm_0$ and $\tilde{T}^\pm_1$
		of $\tilde{J}^\pm$ covering $T$
		such that $H_1(T^\pm_0;\Rational)\oplus H_1(T^\pm_1;\Rational)$ 
		has at least $3$ dimensions surviving in $H_1(\tilde{J}^\pm;\Rational)$.
		Without loss of generality, we may assume $H_1(T_1^\pm;\Rational)$
		to be embedded in $H_1(\tilde{J}^\pm;\Rational)$.
		For $i$ equal to $0$ or $1$, denote by 
			$$\partial_i:H_2(\tilde{J}^\pm,\partial\tilde{J}^\pm;\Integral)\to H_1(\tilde{T}^\pm_i;\Integral)$$
		the boundary homomorphism followed by a projection to 
		$H_1(\tilde{T}^\pm_i;\Integral)$.
		Choose a generic pair of positive integers $r^\pm$ subject to
			$$r^+-r^-=-k,$$
		then there exists $\beta^\pm$ in $H_2(\tilde{J}^\pm,\partial\tilde{J}^\pm;\Integral)$ such that
			$$\partial_0\beta^\pm\,=\,0,\textrm{ and }\partial_1\beta^\pm\,=\,p\cdot([\tilde{l}^\pm]\,+\,r^\pm\cdot[\tilde{e}]),$$
		for some positive integer $p$,
			%$$\partial_1\beta^\pm\,=\,p\cdot([\tilde{l}^\pm]\,+\,r^\pm\cdot[\tilde{e}]).$$
		where $\tilde{l}^\pm$ and $\tilde{e}$ stands for  corresponding elevated curves in $\tilde{T}_1^\pm$.
		In fact, $\beta^\pm$ can be constructed dually 
		by taking an (integral) homomorphism $H_1(\tilde{J}^\pm;\Rational)\to \Rational$
		whose kernel contains the image of $H_1(\tilde{T}_0^\pm;\Rational)$ and intersects $H_1(\tilde{T}_1^\pm;\Rational)$
		in the subspace spanned by the vector $p\cdot([\tilde{l}^\pm]\,+\,r^\pm\cdot[\tilde{e}])$.
		Denote by $\tilde{F}^\pm$ an elevation of $F^\pm$ in $\tilde{J}^\pm$.  
		Then for a positive integer $q$, there is an oriented taut subsurface
		$\tilde{E}^\pm$ of $\tilde{J}^\pm$, whose relative homology class satisfies
			$$[\tilde{E}^\pm]\,=\,\beta^\pm\,+\,q\cdot[\tilde{F}^\pm].$$
		Moreover, choosing $q$ to be sufficiently large, 
		$[\tilde{E}^\pm]$ lies in the fibered cone of the Thurston norm containing $[\tilde{F}^\pm]$.
		This guarantees that $\tilde{E}^\pm$ can be arranged
		to be transverse to the Nielsen--Thurston suspension flow
		in $J^\pm$, in the periodic case since $\tilde{E}^\pm$ is horizontal,
		or in the pseudo-Anosov case by Fried \cite{Fried}. 
		Let $E^\pm$ be the properly essentially immersed subsurface in $J^\pm$ projected from $\tilde{E}^\pm$.
		It follows that $E^\pm$ is also transverse to the Nielsen--Thurston suspension flow
		in $J^\pm$.
		Moreover, there are two components $\partial E^\pm$,
		namely, $c^\pm_0$ and $c^\pm_1$, which are immersed in $T^\pm$ projected from
		components of $\partial_0\tilde{E}^\pm$ and $\partial_1\tilde{E}^\pm$, respectively.
		Note that $c^+_0$ and $c^-_0$ represent the same immersed curve in $T$ up to homotopy,
		which we denote as $c_0$, and similarly denote $c_1$ for the immersed $c^+_1$ and $c^-_1$.
		
		Let $V$ be the union of $J^+$ and $J^-$, glued along $T$. Glue up 
		$E^+$ and $E^-$ along $c_0$ and $c_1$.
		The result is a partially immersed subsurface 
		$E$ of $V$, in the sense that all boundary components of $E$ lie on a JSJ torus 
		in the interior or on the boundary of $V$.
		Take two oppositely oriented copies $E$ and $\bar{E}$ of the partially immersed
		subsurface $E$.
		The components of $\partial E$ and $\partial\bar{E}$ are
		oppositely directed pairs of curves $\{b_i,\bar{b}_i\}$. For each pair $\{b_i,\bar{b}_i\}$,
		construct an essentially immersed, virtually embedded, oriented compact subsurface $R_i$,
		bounded by an equidegree finite cover of $b_i\sqcup\bar{b}_i$, 
		locally from the side opposite to $E\sqcup\bar{E}$.
		This can be done by taking $R_i$ to be a partial PW subsurface 
		as discussed in \cite[Section 4]{DLW}.
		Merge $R_i$ and $E\sqcup\bar{E}$ along boundary.
		In other words, pass to some finite cover
		of each subsurface	and glue them up along the matched boundary components,
		(cf.~\cite[Subsection 3.3]{PW-mixed}).
		As a result, we obtain an essentially immersion of a closed oriented
		surface 
			$$j:S\looparrowright M,$$
		together with an essential subsurface $E'$ of $S$ having the following property:
		The restriction of  $j$ to $E'$ factorizes
		as a covering followed by the partial immersion:
			$$E'\to E\to V.$$
		
		Observe that $\pfib(S)$ is transverse to the Nielsen--Thruston suspension flow. Indeed,
		over $E\sqcup\bar{E}$ this follows from our construction; 
		over $R_i$ this is because
		a partial PW subsurface by definition does not contain geometrically infinite JSJ
		subsurfaces carried by hyperbolic pieces.
				
		To see that $S$ is not aspiral, we take a directed loop $\alpha$ in $E$ which crosses
		$c_0$ and $c_1$ once each, assuming that $\alpha$ crosses $c_1$ 
		from the $-$ side to the $+$ side.
		Let $\alpha'$ be an elevation of $\alpha$ in $E'$.
		Hence $\alpha'$ is a loop in the almost fiber part
		$\pfib(S)$. From the construction, $\sigma_M(c_0)=1$ and $\sigma_M(c_1)=(pr^-+q)/(pr^++q)$.
		By Lemma \ref{equiperiodic} and our assumption of $\theta$,
		the correction factor $\rho_M(\alpha')=1$.
		Then by Formula \ref{formulaRW}:
				$$s(\alpha')\,=\,\left(\frac{pr^-+q}{pr^++q}\right)^{[\alpha':\alpha]}\,\neq\,1.$$
		We conclude that $\pfib(S)$ is not aspiral,
		so $S$ is not virtually embedded by Theorem \ref{main-aspiralityCriterion}.
	\end{proof}

\section{Decision problems}\label{Sec-decisionProblems}
	
		As the last part of our discussion, we consider the decision problem for surface subgroup separability.
		We refer the readers to a recent survey \cite{AFW-decision} of Aschenbrenner--Friedl--Wilton
		for background about decision problems in $3$-manifold groups.
		
		\begin{proposition}\label{groupDecision}
			The surface subgroups of a finitely presented closed $3$-manifold group
			are recursively enumerable, 
			by exhibiting a finitely presented generating set for each of them.
			Moreover, the separability of these subgroups can be decided. 
		\end{proposition}
		
		As the Membership Problem is solvable for finitely presented $3$-manifold groups
		\cite[Theorem 4.11]{AFW-decision},
		Proposition \ref{groupDecision} implies a solution to
		the decision problem for the separability of a (presumed) surface subgroup
		in a finitely presented $3$-manifold group, as asserted by 
		Corollary \ref{decideSeparability}. 
		
		We first solve a topological decision problem.
		
		\begin{lemma}\label{PLDecision}
			If $S$ is a piecewise linearly immersed triangulated closed subsurface of a closed triangulated $3$-manifold $M$,
			then it is decidable whether or not $S$ is essentially immersed, and if yes, whether or not $S$ is virtually embedded.
		\end{lemma}
		
		\begin{proof}
			We may first check if $M$ is orientable by computing its orientation character, 
			and if $M$ is not orientable, we produce a double cover of $M$ and choose an elevation of $S$.
			Without loss of generality, we may assume that $M$ is orientable and $S$ is not a sphere.
			
			Find a collection of essential spheres of $M$ which yields a Kneser--Milnor decomposition 
			of $M$. Simplify the intersection of $S$ with the essential spheres by disk surgeries
			to produce an immersed closed subsurface $S'$ in the complement.
			If $S'$ is not homeomorphic to $S$ together with possibly some sphere components,
			then $S$ is not essentially immersed. 
			Otherwise, replace $M$ with the closed Kneser--Milnor factor containing 
			the non-sphere component of $S'$. 
			Thus we may assume that $M$ is irreducible.
			
			Proceed to find a collection of of essential tori of $M$ which gives rise to a JSJ decomposition of $M$.
			If $S$ intersects a JSJ torus in a contractible loop which is homotopically nontrivial on $S$,
			then $S$ is not essentially immersed. 
			Otherwise, simplify the position of 
			$S$ by disk surgeries and discard sphere components. Further simplify the position of $S$ 
			by annulus surgeries to obtain the induced JSJ decomposition of $S$.
			
			Note that $S$ is essentially immersed if and only if the JSJ subsurfaces $S_v$ are essentially immersed
			in the carrier JSJ pieces $J$ of $M$. 
			For each $S_v$, simutaneously run the following three processes:
			\begin{itemize}
				\item Recursively enumerate all the nontrivial elements of $\pi_1(S_v)$;
				at each step, decide whether the element is trivial in $\pi_1(J)$; return if yes. 
				\item Recursively enumerate all the elements of $\pi_1(J)$; at each step,
				decide whether the element is a normalizer of $\pi_1(S_v)$ 
				and not contained in $\pi_1(S_v)$; return if yes.
				\item Recursively enumerate all the homomorphisms
				from any finite index subgroups of $\pi_1(J)$ to $\pi_1(S_v)$;
				at each step, decide whether the homomorphism is a retraction
				to the image of $\pi_1(S_v)$; return if yes.
			\end{itemize}
			One of the three processes must halt, so we have certified either that $S_v$ is not essentially 
			properly immersed, or that $S_v$ is a virtual fiber, or that $S_v$ is a virtual retract.
			
			In the first case, $S$ is not virtually immersed. In the second or the third case,
			$S$ is virtually immersed, and we have recognized the almost fiber part $\pfib(S)$.
			Moreover, we have also obtained a collection of the setup
			data which defines the associated $\Rational^\times$-principal bundle $\mathscr{H}$
			(Subsection \ref{Subsec-bundleH}) by the second process. 
			Then we can compute the spirality character of the almost fiber part $\pfib(S)$
			using Formula \ref{spiralityH},
			and decide whether or not it always takes values in $\{\pm1\}$.
			Respectively, we can conclude whether or not $S$ is virtually embedded.
		\end{proof}
		
		%\begin{lemma}
			%Let $\pi$ be a finitely presented subgroup of a closed $3$-manifold group
			%$\pi_1(M)$,	and $H$ be a closed surface subgroup of $\pi$. Suppose that $\pi$ is a retract
			%of $\pi_1(M)$. Then $H$ is separable in $\pi$  
			%if and only if $H$ is separable in $\pi_1(M)$.
		%\end{lemma}
		%
		%\begin{proof}
			%It suffices to prove the `only if' direction. Suppose that $H$ is separable in $\pi$.
			%By Scott \cite{Sc-core}, the covering space $\tilde{M}$ of the closed 3-manifold
			%$M$ is homotopy equivalent to a compact $3$-submanfold $K$, and $H$ can be realized
			%by an essentially immersed closed subsurface $S$ of $K$. By Przytycki--Wise \cite{PW-embedded},
			%$S$ is virtually embedded in $K$.
			%
		%\end{proof}
		%
		\begin{proof}[{Proof of Proposition \ref{groupDecision}}]
			Let $\mathcal{P}$ be a finite presentation of a closed $3$-manifold group $\pi$.
			There is an algorithm to produce a pointed closed triangulated $3$-manifold $(M,x_0)$
			with an isomorphism $\pi_1(M,x_0)\cong\pi$, \cite[Lemma 3.3]{AFW-decision}. 
			
			Consider any pair $(S,h)$ where $S$ is a triangulated closed surface with a base point $*$,
			and $h:\pi_1(S,*)\to\pi_1(M,x_0)$ is a homomorphism. Possibly after barycentric subdivision,
			we can realize $h$ by a pointed piecewise linear 
			map $S\to M$. By Lemma \ref{PLDecision}, we can decide whether $S$ is
			essentially immersed, and if yes, whether it is virtually embedded.
			In other words, we can decide whether $h$ is injective, and if yes,
			whether the image of $h$ is separable by Przytycki--Wise \cite{PW-embedded}.
			When $h$ is injective, we return with 
			a finitely presented generating set of the image of $h$,
			which is induced by the finite presentation of $\pi_1(S,*)$.
			Otherwise, we return with no output.
			
			Recursively enumerate all the pairs $(S,h)$, then the above process
			outputs a complete list of all the surface subgroups of $\pi$,
			each given by a finitely presented generating subset, 
			and decorated with a mark of separability.
		\end{proof}
		
	\section{Conclusions}\label{Sec-conclusions}
		In conclusion, essentially embedded closed subsurfaces in closed orientable aspherical $3$-manifolds
		can be virtually fit into a globally interesting position,
		and the obstruction 
		for essentially immersed subsurfaces
		to having such virtual arrangement
		is the spirality character of the almost fiber part. 
		When a subsurface is locally nicely immersed with respect to a fibered cone,
		the spirality character can be determined more effectively.
		
		We propose some further questions regarding suspension flow and essentially
		immersed subsurfaces. Let $S$ be a closed essentially immersed subsurface of a closed
		orientable nonpositively curved $3$-manifold $M$ ---
		the assumption on $M$ is to heuristically ensure
		many flexible virtual fibrations \cite{Agol-RFRS,Agol-VHC,Liu,PW-mixed}.
		
		\begin{question}
			If $S$ has no JSJ subsurface 
			which is a vertical annulus carried by a Seifert fibered piece,
			does $S$ virtually satisfy 
			the assumption of Formula \ref{formulaRW}?			
		\end{question}
						
		\begin{question}
			Under the assumption of Formula \ref{formulaRW},
			do we have an effective characterization of the almost fiber part $\pfib(S)$
			analogous to the finite foliation criterion of
			Cooper--Long--Reid \cite{CLR-finiteFoliation,CLR-SIET}?			
		\end{question}
				
		\begin{question}
			How to construct a virtual fiber of $M$ with control on the degeneracy slopes?
		\end{question}

\bibliographystyle{amsalpha}

%\bibliography{../refs}

\end{document}